\documentclass[reqno]{amsart}
\usepackage{amsthm,amsmath,amssymb}
\usepackage[pdftex]{graphicx}
\usepackage{braket}
\usepackage{color}
\usepackage{shadethm}
\usepackage{amscd}
\usepackage{mathrsfs}
\usepackage{comment}

\title{The $\nu^+$-equivalence classes of genus one knots}
\author{Kouki Sato}
\date{}

\newtheorem*{question}{Question}
\newtheorem{thm}{Theorem}[section]
\newtheorem{prop}[thm]{Proposition}
\newtheorem{lem}[thm]{Lemma}
\newtheorem{cor}[thm]{Corollary}
\newtheorem{claim}{Claim}

\theoremstyle{definition}

\newtheorem*{dfn}{Definition}
\newtheorem*{condition}{Condition}

\newtheorem*{remark}{Remark}
\newtheorem*{acknowledge}{Acknowledgements}

\DeclareMathOperator{\Spin}{Spin}

\DeclareMathOperator{\Int}{Int}
\DeclareMathOperator{\image}{Im}
\DeclareMathOperator{\nuplus}{\overset{\nu^+}{\sim}}

\DeclareMathOperator{\lk}{lk}

\DeclareMathOperator{\R}{\mathbb{R}}

\DeclareMathOperator{\z}{\mathbb{Z}}
\DeclareMathOperator{\Alex}{Alex}
\DeclareMathOperator{\Alg}{Alg}
\DeclareMathOperator{\gr}{gr}
\DeclareMathOperator{\F}{\mathbb{F}}
\DeclareMathOperator{\mF}{\mathcal{F}}
\DeclareMathOperator{\mKf}{\mathcal{K}^{\mathit{f}}}
\DeclareMathOperator{\mK}{\mathcal{K}}
\DeclareMathOperator{\mCf}{\mathcal{C}^{\mathit{f}}}
\DeclareMathOperator{\mC}{\mathcal{C}}
\DeclareMathOperator{\spanF}{span_{\F}}
\DeclareMathOperator{\spanFU}{span_{\F[U]}}
\DeclareMathOperator{\coker}{coker}
\DeclareMathOperator{\PL}{PL}
\DeclareMathOperator{\mdeg}{mdeg}
\DeclareMathOperator{\Mdeg}{Mdeg}

\DeclareMathOperator{\Hom}{Hom}

\DeclareMathOperator{\spanL}{span_{\Lambda}}
\DeclareMathOperator{\tG}{\widetilde{\mathcal{G}}}
\DeclareMathOperator{\G}{\mathcal{G}}
\DeclareMathOperator{\CR}{\mathcal{CR}}
\DeclareMathOperator{\mhF}{\mathcal{\widehat{F}}}
\DeclareMathOperator{\hC}{\mathit{\widehat{C}}}

\newcommand{\falex}[1]{\mF^{\Alex}_{#1}}
\newcommand{\falg}[1]{\mF^{\Alg}_{#1}}
\newcommand{\ot}[1]{\otimes_{#1}}
\def\np{\nu^+}
\def\bC{\bar{C}}
\def\bpartial{\bar{\partial}}
\def\mB{\mathcal{B}}
\def\sR{\mathscr{R}}
\def\bF{\bar{\mathcal{F}}}

\def\gen{\mathrm{gen}}
\def\tgen{\widetilde{\gen}}

\begin{document}

\begin{abstract}
The $\nu^+$-equivalence is an equivalence relation on the knot concordance group. 
This relation can be seen as a certain stable equivalence on knot Floer complexes $CFK^{\infty}$, and many concordance invariants derived from Heegaard Floer theory are invariant under the equivalence. In this paper, we show that any genus one knot is 
$\nu^+$-equivalent to one of the trefoil, its mirror and the unknot.
\end{abstract}
\maketitle

\tableofcontents

\section{Introduction}

Throughout this paper, all manifolds are assumed to be smooth, compact, connected, orientable and oriented unless otherwise stated.

\subsection{Back grounds and the main theorem}

Heegaard Floer homology \cite{OS04def} is a powerful set of invariants for 3- and 4-manifolds and knots in 3-manifolds. 
In Particular, the $\mathbb{Z}^2$-filtered chain complex $CFK^{\infty}(K)$ 
\cite{OS04knot}
associated to any knot $K$ in $S^3$ is 
a very effective tool in studying
knots and Dehn surgeries along knots. Indeed, 
from $CFK^{\infty}(K)$, we can compute
\begin{itemize}
\item the knot Floer homology $\widehat{HFK}(K)$ \cite{OS04knot}, 
and so we can detect the genus and fibredness of $K$ 
\cite{Ju08, Ni07, OS04genus},
\item the Floer homology groups $\widehat{HF}$, $HF^{\infty}$ and $HF^{\pm}$
and correction terms $d(-, \mathfrak{s})$ of all Dehn surgeries along $K$
\cite{OS08integer, OS11rational}, and
\item many knot concordance invariants including 
$\nu^+$, $\tau$, $\Upsilon$, $\Upsilon^2$, and so on. 
(See \cite{Ho17, KL18}
for details.)
\end{itemize}
In this paper,
to improve the understanding of $CFK^{\infty}$, we study
{\it $\nu^+$-equivalence} (denoted $\nuplus$)
introduced by 
Hom \cite{Ho17}
 and 
Kim-Park \cite{KP18}.
Here, two knots $K_1$ and $K_2$ are {\it $\nu^+$-equivalent}
if $\nu^+(K_1\#(-K_2^*))=\nu^+(K_2\#(-K_1^*))=0$,
where $-K$ and $K^*$ denote the inverse and the mirror of $K$ respectively,
and $\nu^+$ is a $\Bbb{Z}_{\geq0}$-valued concordance invariant defined by 
Hom-Wu 
\cite{HW16}. 
This relation is an equivalence relation on knots, and
if two knots are concordant then they are $\nu^+$-equivalent.
(We call the equivalence classes {\it $\np$-classes}.)
By the following Hom's theorem, 
$\nu^+$-equivalence can be seen as a `stable' filtered chain homotopy equivalence on $CFK^{\infty}$.
\begin{thm}[\cite{Ho17}]
Two knots $K_1$ and $K_2$ are $\nu^+$-equivalent
if and only if we have the following $\z^2$-filtered chain homotopy equivalence:
$$
CFK^{\infty}(K_1) \oplus A_1 \simeq CFK^{\infty}(K_2) \oplus A_2,
$$
where $A_1$,$A_2$ are acyclic, i.e., $H_*(A_1)=H_*(A_2)=0$.
\end{thm}
This theorem shows that determining the $\nu^+$-class of knots is meaningful
in terms of $CFK^{\infty}$.
Moreover, the $\nu^+$-class of a knot $K$ determines
all correction terms of all Dehn surgeries along $K$ and many concordance invariants
including $\np$, $\tau$, $\Upsilon$ and $\Upsilon^2$ of $K$,
and hence classifying the $\nu^+$-classes is useful for computing these invariants.

The aim of this paper is to classify the $\nu^+$-classes of genus one knots
by using the $\tau$-invariant \cite{OS03tau};
in fact, we found that only three $\nu^+$-classes are realized by genus one knots. To state our theorem, we set some notations.
For any knot $K$, let $[K]_{\nu^+}$ denote the
$\nu^{+}$-class of $K$ and $g(K)$ the genus of $K$.
For coprime positive integers $p$ and $q$, let $T_{p,q}$ denote the $(p,q)$-torus knot.
\begin{thm}
\label{main thm}
For any knot $K$ with $g(K)=1$,
we have
$$
[K]_{\nu^+}
=
\begin{cases}
\left[ T_{2,3} \right]_{\nu^+} & \text{if } \tau(K)=1\\
\left[\text{{\rm unknot}}\right]_{\nu^+} & \text{if } \tau(K)=0\\
\left[ (T_{2,3})^* \right]_{\nu^+} & \text{if } \tau(K)=-1
\end{cases}.
$$
In other words, any genus one knot is $\nu^+$-equivalent
to one of the trefoil, its mirror and the unknot.
\end{thm}
Since the $\tau$-invariant is relatively understood,
Theorem \ref{main thm} enables us to determine the $\nu^+$-class of many
concrete examples.
For instance, Hedden \cite{He07} gives a formula 
for the $\tau$-invariant of the positive $t$-twisted Whitehead double of a knot $K$
(denoted by $D_+(K,t)$).
By Theorem \ref{main thm},
we can generalize his formula to a formula for the $\nu^+$-class
of $D_+(K,t)$.

\begin{cor}
\label{Wh}
For any knot $K$ and $t \in \z$, we have
$$
[D_+(K,t)]_{\nu^+}
=
\begin{cases}
[\text{{\rm unknot}}]_{\nu^+} & \text{for } t \geq 2\tau(K)\\
\left[ T_{2,3} \right]_{\nu^+} & \text{for } t < 2\tau(K)
\end{cases}.$$
\end{cor}

Next, let us consider the quotient set $\mathcal{C}_{\nu^+}:=\{ \text{knots in }S^3\}/\nuplus$.
Note that since $\nuplus$ is weaker than knot concordance and the $\nu^+$-invariant has the sub-additivity, 
we can identify $\mathcal{C}_{\nu^+}$ with a quotient group of the knot concordance group $\mathcal{C}$. So
it is natural to ask how different these groups are.
To give an observation of the question, 
we set $\mathcal{F}_n$ to be the subgroup of $\mathcal{C}$ generated by
the knots with genus at most $g$.
Let $\pi_{\np} \colon \mathcal{C} \to \mathcal{C}_{\nu^+}$ be the projection,
and then the sequence $\{\mathcal{F}_g\}_{g\in \mathbb{Z}_{\geq 0}}$ gives filtrations

$$0= \mathcal{F}_0 \subset \mathcal{F}_1 \subset \mathcal{F}_2 
\subset \cdots \subset \mathcal{C}$$
and
$$0= \pi_{\np}(\mathcal{F}_0) \subset \pi_{\np}(\mathcal{F}_1) 
\subset \pi_{\np}(\mathcal{F}_2)
\subset \cdots \subset \mathcal{C}_{\nu^+}.$$
It is easy to show that $\mathcal{F}_1$ contains $\z^{\infty}$ as a summand.
(For instance, compute the $\omega$-signature for the ``twisted doubles'' of the unknot. We refer to \cite{Li97}.) 
Therefore, combining it with Theorem~\ref{main thm},
we have the following proposition, which shows a big gap between $\mathcal{C}$
and $\mathcal{C}_{\nu^+}$. 
\begin{prop}
\label{filter}
$\mathcal{F}_1$ contains $\mathbb{Z}^{\infty}$ as a summand, while
$\pi_{\np}(\mathcal{F}_1)$ is isomorphic to $\Bbb{Z}$.
\end{prop}
In knot concordance theory, there are few kinds of filtrations with
each level finitely generated. Hence we suggest the following question.

\begin{question}
For each $g \in \z_{\geq 0}$, is $\pi_{\np}(\mathcal{F}_g)$ finitely generated?
\end{question}

\subsection{The idea of proof: Estimating $\nu^+$-classes}
In order to prove Theorem \ref{main thm}, we use a partial order on
$\mathcal{C}_{\nu^+}$ (denoted $\leq$) introduced in the author's paper \cite{Sa18full}.
We first study this partial order geometrically to give the following estimate for the $\nu^+$-class of any knot $K$.
Here $g_4(K)$ denotes the 4-genus of $K$, and we note that this estimate depends on $g_4(K)$ rather than $g(K)$.
\begin{thm}
\label{geom}
For any knot $K$, we have
$$
-g_4(K)[T_{2,3}]_{\nu^+} \leq [K]_{\nu^+} \leq g_4(K)[T_{2,3}]_{\nu^+}.
$$
\end{thm}
Next, we study the $\Bbb{Z}^2$-filtered structure of $CFK^{\infty}$
 with $g(K)=1$ algebraically to obtain another estimate,
and combine it with Theorem \ref{geom} to prove Theorem \ref{main thm}.
As another consequence of such estimates, we have the following discriminant
using the $\Upsilon$-invariant \cite{OSS17}.
\begin{thm}
\label{discriminant}
The equality $[K]_{\np}= -g(K)[T_{2,3}]_{\nu^+}$ holds
if and only if $\Upsilon_{K}(1)= g(K)$.
\end{thm}

\subsection{Formal knot complexes and new concordance invariants}
To study the algebraic aspects of $\np$-classes deeply,
we consider an algebraic generalization of $CFK^{\infty}$ called 
{\it formal knot complexes}.
(The notion is originally considered in \cite{KL18}.)
In particular, we establish the category of such complexes, and obtain 
{\it the formal knot monoid} $\mathcal{K}$ and  {\it the formal knot concordance group}
$\mCf$, which are analogies of the knot monoid $\mathcal{K}$ and 
the knot concordance group $\mathcal{C}$, respectively.
Concretely, these monoids are related as follows.
\begin{thm}
\label{comm diag}
We have the following commutative diagram:
$$
\begin{CD}
\mK @>  [K] \mapsto [CFK^{\infty}(K)] >> \mKf \\
@V [K] \mapsto [K]_c VV @VV [C] \mapsto [C]_{\nu^+} V\\
\mC @>> [K]_c \mapsto [CFK^{\infty}(K)]_{\nu^+} > \mCf
\end{CD}
$$
Here, the bottom map coincides with $\pi_{\np}$. In particular, the image of the bottom map is $\mC_{\np}$.
\end{thm}
Moreover, we also introduce the genus of formal knot complexes, and 
define the genus filtration
$$
0=\mF^f_0 \subset \mF^f_1 \subset \mF^f_2 \subset \cdots \subset \mCf,
$$
where $\pi_{\np}(\mF_g) \subset \mF^f_g$.
For example, Figure~\ref{C^n} depicts an infinite family of genus one 
formal knot complexes, and hence $[C^n]_{\np} \in \mF^f_1$ for each 
$n \in \z_{>0}$.
Here we note that $C^1$ is $CFK^{\infty}(T_{2,3})$.
\begin{figure}[htbp]
\begin{center}
\includegraphics[scale= 0.3]{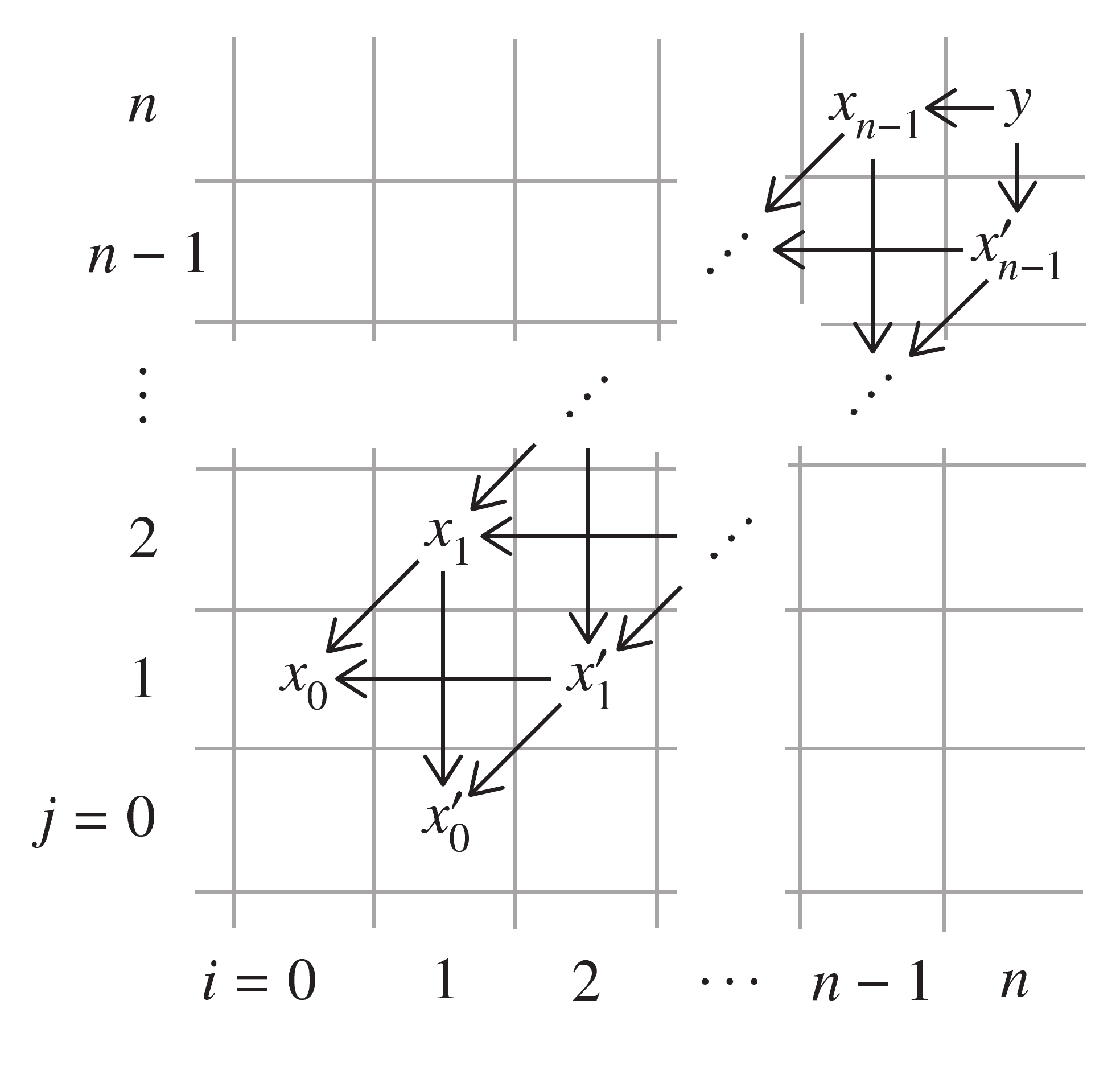}
\vspace{-4mm}
\caption{A formal knot complex $C^n$ with genus one.}
\label{C^n}
\end{center}
\end{figure}
We prove that the $[C^n]_{\np}$ are mutually distinct, which implies 
that Theorem~\ref{main thm} cannot be proved purely algebraically.
\begin{thm}
\label{formal genus1}
The $\np$-classes $\{[C^{n}]_{\np}\}_{n=1}^{\infty}$ are mutually distinct in $\mCf$,
while $\tau(C^n)=1$ for any $n$. In particular, the complement $\mF^f_1 \setminus \pi_{\np}(\mF_1)$ is infinite.
\end{thm}
In addition, we will show that if a formal knot complex $C$ is realized as $CFK^{\infty}$ for some knot $K$,
then the genus of $C$ is at least $g(K)$. Since $C^n$ has genus one and 
$\tau(C^n)=1$ but
cannot be realized by any genus one knot, we have the following result,
which is related to the geography problem discussed in \cite{HW18}.
\begin{cor}
\label{formal genus1 cor}
The formal knot complexes $\{C^{n}\}_{n = 2}^{\infty}$ 
cannot be realized by any knot in $S^3$.
\end{cor}
In order to distinguish the complexes $\{C^n\}$, we introduce an infinite family 
$\{\G_k\}_{k=0}^{\infty}$ of invariants of $\np$-classes, where
$\G_k(C)$ consists  of finitely many subsets of $\z^2$.
Since the $\np$-class of knots is a knot concordance invariant, 
$\{\G_n\}_{n=0}^{\infty}$ also gives new knot concordance invariants. 
In particular, the primary invariant $\G_0$ has the following properties.
\begin{thm}
\label{G_0 properties}
For any knot $K$, the following assertions hold:
\begin{enumerate}
\item $\G_0(K)$ determines all correction terms of all Dehn surgeries along $K$.
\item $\G_0(K)$ determines all of $\np$, $\tau$ and $\Upsilon$.
\item $[K]_{\np}=0$ if and only if $\G_0(K)$ has $\{(i,j) \in \z^2 \mid i \leq 0, j \leq 0\}$ as the unique element.
\end{enumerate}
\end{thm}
The definition of $\G_n$ and explicit formulas for computing the above invariants from 
$\G_0(K)$ is given in Section~5. 
In the section, we also discuss the relationship between our secondary invariant $\G_1$ and
the $\Upsilon^2$-invariant \cite{KL18}.
\subsection*{Organization}
In Section~2, we establish the category of formal knot complexes, and construct the monoid $\mathcal{K}^f$ and the abelian group $\mCf$. Theorem~\ref{comm diag} is also proved in this section.
In Section~3, we prove Theorem~\ref{geom}. In Section~4, we discuss algebraic estimates for $\np$-classes, and prove Theorem~\ref{main thm} and Theorem~\ref{discriminant}.
In Section~5, we introduce the invariants $\{\G_n\}$, and prove 
Theorem~\ref{formal genus1}, Corollary~\ref{formal genus1 cor} and 
Theorem~\ref{G_0 properties}.

\begin{acknowledge}
The authors would like to thank Jennifer Hom, Min Hoon Kim and JungHwan Park 
for many interesting conversations about the present work.
The author was supported by JSPS KAKENHI Grant Number 18J00808.
\end{acknowledge}


\section{Category of  formal knot complexes}
In this section, we establish the category of formal knot complexes.
\subsection{Poset filtered chain complexes}
Let $P$ be a {\it poset}, i.e.\ a set $P$ with partial order $\leq$.
For example, we often consider the partial order $\leq$ on $\z^2$ given by
 $(i,j)\leq(k,l)$ if $i \leq k$, $j \leq l$.
For a given poset $P$, a \textit{closed region} $R \subset P$ is a subset
such that for any $x \in P$, if there exists an element $y \in R$ satisfying $x \leq y$, then $x \in R$. We denote the set of closed regions of $P$ by $\CR(P)$.

Let $\F := \z /2\z$ and $\sR$ be a $\F$-algebra.
In this paper, we say that
$(C,\partial)$ is a {\it chain complex $C$ over $\sR$}
if $(C,\partial)$ satisfies the following:
\begin{itemize}
\item $C$ is a $\sR$-module and $\partial \colon C \to C$ is a
$\sR$-linear map with $\partial \circ \partial = 0$. 
\item As $\F$-vector space, $C$ is decomposed into $\bigoplus_{n \in \z} C_n$ and satisfies $\partial(C_n) \subset C_{n-1}$.
\end{itemize}
(Remark that the $\sR$-action does not preserve the grading in general.
We often abbreviate $(C,\partial)$ to $C$.)
Then, we say that $C$ is \textit{$P$-filtered} if 
a subcomplex $C_R$ of $C$ over $\F$ is associated to each closed region $R \subset P$
so that if $R \subset R'$ then $C_R \subset C_{R'}$. 
(Here we remark that $C_R$ is not a $\sR$-submodule of $C$ in general.)
We call the set $\{C_R\}_{R \in \CR(P)}$ a {\it $P$-filtration} on $C$.
For instance, a $\z$-filtration $\{C_{\{i \leq m\}}\}_{\{i \leq m\} \in \CR(\z)}$ is 
identified with an increasing sequence
$$
0 \subset \cdots \subset  \mF_m \subset \mF_{m+1} \subset \cdots \subset C
$$
of subcomplexes by the correspondence $\mF_m=C_{\{i \leq m\}}$.
Moreover, For two $\z$-filtrations $\{\mF^1_i\}_{i \in \z}$ and 
$\{\mF^2_j\}_{j \in \z}$ on $C$, the set
$$
\{C_R\}_{R \in \CR(\z^2)} := \{\sum_{(i,j)\in R} \mF^1_i \cap \mF^2_j\}_{R \in \CR(\z^2)}
$$
defines a $\z^2$-filtration on $C$.
We call it {\it the $\z^2$-filtration induced by the ordered pair} 
$(\{\mF^1_i\}_{i \in \z}, \{\mF^2_j\}_{j \in \z})$. 
For a complex $C$ with induced $\z^2$-filtration by
$(\{\mF^1_i\}, \{\mF^2_j\})$, 
$C^r$ denotes $C$ with induced $\z^2$-filtration by $(\{\mF^2_i\}, \{\mF^1_j\})$.

For any two $P$-filtered chain complexes $C$ and $C'$, a map 
$f: C \to C'$ is \textit{$P$-filtered}
if $f(C_R) \subset C'_R$ for any closed region $R$.
Two $P$-filtered chain complexes $C$ and $C'$ are 
\textit{$P$-filtered homotopy equivalent}
(and denoted $C \simeq C'$)
if there exists a chain homotopy equivalence map $f: C \to C'$ over $\sR$ 
such that the map, its inverse and all chain homotopies are $P$-filtered
and graded. (Then $f$ is called 
a \textit{$P$-filtered homotopy equivalence map}.
Particularly, we call the above $f$ a 
\textit{$P$-filtered chain isomorphism} if $f$ is a chain isomorphism.)
The following lemma immediately follows from the definition of 
$P$-filtered homotopy equivalence.
\begin{prop}
\label{exact}
Let $C$ and $C'$ be $P$-filtered chain complexes.
If $C \simeq C'$, then for any closed regions $R \subset R'$,
we have an isomorphism between the long exact sequences of $\sR$-modules:
$$
\begin{CD}
\cdots @>{\partial_{*}}>> H_n(C_R) 
@>{i_{*}}>> H_*(C_{R'}) @>{p_{*}}>> H_*(C_{R'}/C_{R}) 
@>{\partial_{*}}>>  \cdots \\
@.   @V{\cong}VV  @V{\cong}VV @V{\cong}VV  @. \\
\cdots @>{\partial_{*}}>> H_*(C'_R) @>{i_{*}}>> H_*(C'_{R'}) 
@>{p_{*}}>> H_*(C'_{R'}/C'_{R}) @>{\partial_{*}}>> \cdots \\
\end{CD}
$$
Here, $i \colon C_R \to C_{R'}$ (resp.\ $p \colon C_{R'} \to C_{R'}/C_R$)
denote the inclusion (resp.\ the projection).
Moreover, the above isomorphism induces an isomorphism  
between the long exact sequences of graded $\F$-vector spaces:
$$
\begin{CD}
\cdots @>{\partial_{*,n+1}}>> H_n(C_R) 
@>{i_{*,n}}>> H_n(C_{R'}) @>{p_{*,n}}>> H_n(C_{R'}/C_{R}) 
@>{\partial_{*,n}}>>  \cdots \\
@.   @V{\cong}VV  @V{\cong}VV @V{\cong}VV  @. \\
\cdots @>{\partial_{*,n+1}}>> H_n(C'_R) @>{i_{*,n}}>> H_n(C'_{R'}) 
@>{p_{*,n}}>> H_n(C'_{R'}/C'_{R}) @>{\partial_{*,n}}>> \cdots \\
\end{CD}
$$
\end{prop}


\subsection{Formal knot complexes}
Now we state the precise definition of formal knot complex, and discuss several basic properties of it.
\subsubsection{Definition}
Let $\Lambda := \F[U, U^{-1}]$.
We call a tuple $$(C, \partial, \{ C_n\}_{n \in \z}, \{\falex{j}\}_{j \in \z}, 
\{\falg{i}\}_{i \in \z})$$
a \textit{formal knot complex} if it satisfies the following seven conditions;
\begin{enumerate}
\item $(C,\partial)$ is a chain complex over $\Lambda$
with decomposition $C= \bigoplus_{n \in \z}C_n$.
The grading of a homogeneous element $x$ is denoted $\gr(x)$ 
and called the \textit{Maslov grading} of $x$.
\item
$\{\falex{j}\}_{j \in \z}$ is a $\z$-filtration on $C$.
This filtration is called \textit{Alexander filtration}, and
the filtration level of an element $x \in C$ is denoted 
$\Alex (x)$ (i.e.\ $\Alex (x) := \min \{ j \mid x \in \falex{j} \}$).

\item
Similarly, $\{ \falg{i} \}_{i \in \z}$ is a $\z$-filtration on $C$, 
 called the \textit{algebraic filtration},
and filtration levels of elements are denoted $\Alg(x)$.
When we regard $C$ as a $\z^2$-filtered complex, we use 
the $\z^2$-filtration induced by the ordered pair 
$(\{\falg{i}\}_{i \in \z}, \{\falex{j}\}_{j \in \z})$.
\item
The action of $U$ lowers Maslov grading by $2$
and Alexander and algabraic filtration levels by $1$.
\item
As a $\Lambda$-module, $C$ is freely and finitely generated by 
elements $\{x_k\}_{1 \leq k \leq r}$ such that 
\begin{itemize}
\item
each $x_k$ is homogeneous
with respect to the Maslov grading,
\item
$\{ U^{\Alex(x_k)} x_k  \}_{1 \leq k \leq r}$ is a free basis for $\falex{0}$ as a 
$\F[U]$-module, and
\item
$\{ U^{\Alg(x_k)} x_k  \}_{1 \leq k \leq r}$ is a free basis for $\falg{0}$ as a 
$\F[U]$-module.
\end{itemize}
We call such $\{x_k\}_{1 \leq k \leq r}$ a \textit{filtered basis}.

\item
There exists a $\z^2$-filtered homotopy equivalence map $\iota: C \to C^r$.

\item
Regard $\Lambda$ as a chain complex with trivial boundary map, and define 
the Maslov grading by
$$
\Lambda_n =
\left\{
\begin{array}{ll}
\{0, U^{-n/2}\} &(n: \text{ even})\\
0 & (n: \text{ odd})
\end{array}
\right.
$$
and the Alexander and algebraic filtrations by
$$
\falex{i}(\Lambda)=\falg{i}(\Lambda) = U^{-i} \cdot \F[U].
$$
Then there exists a $\z$-filtered homotopy equivalence map
$f_{\Alex}$ (resp.\ $f_{\Alg}$) $:C \to \Lambda$ over $\Lambda$ 
with respect to the Alexander (resp.\ algebraic) filtration.
\end{enumerate}
We often abbreviate the tuple 
$$
(C, \partial, \{ C_n\}, \{\mF^{Alex}_j\}, \{\falg{i}\})
$$
to $C$ or $(C,\partial)$.
\begin{remark}
Note that 
$\{ U^{\Alex(x_k)-j} x_k  \}_{1 \leq k \leq r}$ 
(resp.\ $\{ U^{\Alg(x_k)-i} x_k  \}_{1 \leq k \leq r}$)
is a free basis for $\falex{j}$ (resp.\ $\falg{i}$) as a 
$\F[U]$-module. In particular, the equalities $$
U^k(\falex{j})=\falex{j-k} \text{ and }
U^k(\falg{i})=\falg{i-k}$$ hold for any $i,j,k \in \z$.
(These facts also imply that for any element $x \in C$, both 
$\Alex(x)$ and $\Alg(x)$ are finite.)
Similarly,  
$\{U^{\frac{\gr(x_k)-n}{2}}x_k\}_{k \in [n]}$
is a basis for $C_{n}$ as a $\F$-vector space,
where $[n]$ is a subset of $\{1, \ldots r \}$ consisting of elements with 
$\gr(x_k) \equiv n$ (mod  2), and the equality $U^k(C_n)=C_{n-2k}$ holds.
\end{remark}
As the simplest example,
the tuple 
$$
(\Lambda, \text{zero map}, \{\Lambda_n\}_{n \in \z}, 
\{ \falex{j}(\Lambda)\}_{j \in \z},\{ \falg{i}(\Lambda)\}_{i \in \z})
$$
is a formal knot complex.
In addition, it is easy to see that the following lemmas hold.
\begin{lem}
For any formal knot complex $C$,
the complex $C^r$ is also a formal knot complex.
\end{lem}

\begin{lem}
\label{construction}
Let $(\bC, \bpartial)$ be a chain complex over $\F$ generated by
a finite basis $\{x_k\}_{1 \leq k \leq r}$ with functions 
$$
\Alex \colon \{x_k\}_{1 \leq k \leq r} \to \z 
\text{ and } \Alg \colon \{x_k\}_{1 \leq k \leq r} \to \z
$$
satisfying the following:
\begin{itemize}
\item
The sequences
$$
\bF^{\Alex}_j :=\spanF\{x_k \mid \Alex(x_k) \leq j\} \text{ and }
\bF^{\Alg}_i :=\spanF\{x_k \mid \Alg(x_k) \leq i\}
$$
define $\z$-filtrations on $\bC$, respectively.
\item For the induced $\z^2$-filtration $(\{\bF^{\Alg}_i\}, \{\bF^{\Alex}_j\})$ on $\bC$,
we have a $\z^2$-filtered homotopy equivalence $\bC \simeq \bC^r$.
\item Regard $\F$ as a chain complex over $\F$ with trivial boundary map and grading $\F=\F_0$, and define a $\z$-filtration  by $\bF_i(\F) = \F$ if and only if $i \geq 0$. Then
we have $\z$-filtered homotopy equivalences $\bC \simeq \F$ with respect to
both $\{\bF^{\Alex}_j\}$ and $\{\bF^{\Alg}_i\}$.
\end{itemize}
If we set
\begin{itemize}
\item $C := \bC \otimes_{\F} \Lambda$ and $\partial := \bpartial \otimes 1$,
\item $C_n := \bigoplus_{m \in \z} (\bC_{n+2m} \otimes_{\F} U^m)$, and
\item $\falex{j} := \sum_{m \in \z} (\bF^{\Alex}_{j+m}\otimes_{\F} U^m \F[U])$ and 
$\falg{i} := \sum_{m \in \z} (\bF^{\Alg}_{i+m}\otimes_{\F} U^m \F[U])$,
\end{itemize}
then the tuple
$$
(C, \partial, \{ C_n\}_{n \in \z}, \{\falex{j}\}_{j \in \z}, 
\{\falg{i}\}_{i \in \z})
$$
is a formal knot complex.
\end{lem}

In \cite{OS04knot}, Ozsv\'ath and Szab\'o associate the $\z^2$-filtered homotopy type of a formal knot complex 
$CFK^{\infty}(K)$ to any knot $K$, and prove that it is an isotopy invariant.
To simplify notation,  we write $C^K$ for $CFK^{\infty}(K)$. 
\begin{thm}[\text{\cite{OS04knot}}]
If two knots $K$ and $J$ are isotopic, then $C^K \simeq C^J$.
\end{thm}
Moreover, it is proved that the inverse has the same homotopy type as the original one.
\begin{thm}[\text{\cite{OS04knot}}]
\label{inverse}
For a knot $K$, we have $C^{-K} \simeq (C^K)^r \simeq C^K$.
\end{thm}

\subsubsection{Relationship to abstract infinity complex}
Here, we compare formal knot complex with Hedden-Watson's
{\it abstract infitnity complex}.
First, a {\it graded, bifiltered complex}
is a chain complex over $\F$
which admits a basis $\mB$ with functions:
$$
m \colon \mB \to \z \text{ and } \mF \colon \mB \to \z^2
$$
such that for any $a,b \in \mB$, if the coefficient of $a$ in $\partial b$ is non-zero,
then
$$
m(a) = m(b) -1 \text{ and } \mF(a) \leq \mF(b).
$$
In other words, $C_n := \spanF\{a \in \mB \mid m(a)=n  \}$ ($n \in \z$) defines a grading
and $C_R := \spanF\{a \in \mB \mid \mF(a) \in R \}$ ($R \in \CR(\z^2)$) defines
a $\z^2$-filtration.
\begin{dfn}[\text{\cite[Definition~6.1]{HW18}}]
An {\it abstract infinity complex} is a graded, bifiltered complex
$(C, \partial, \mF)$ satisfying
\begin{enumerate}
\item
$(C, \partial)$ is freely generated as a chain complex over $\Lambda$
by a finite set of graded, bifiltered homogeneous generators.
\item Acting by $U$ shifts the grading by $-2$ and the bifiltration by $(-1,-1)$.
\item $H_*(C,\partial) \cong \Lambda$, where $1 \in \Lambda$ has grading $0$.
\item The complex $(C, \partial, \mF^r)$, where $\mF^r$ is the bifiltration
function $\mF^r(i,j) := \mF(i,j)$, is $\z^2$-filtered homotopy equivalent to 
$(C, \partial, \mF)$.
\end{enumerate}
\end{dfn}

\begin{prop}
Any formal knot complex is an abstract infinity complex.
\end{prop}

\begin{proof}
For a given formal knot complex $C$, we take a filtered basis $\{x_k\}_{1 \leq k \leq r}$
and set
\begin{itemize}
\item
$\mathcal{B} := \{ U^l x_k \mid^{1 \leq k \leq r}_{ l \in \z} \}$,

\item
$m: \mathcal{B} \to \z : U^l x_k \mapsto \gr(U^l x_k)$,  and 

\item
$\mathcal{F}: \mathcal{B} \to \z \times \z: U^l x_k \mapsto (\Alg(U^l x_k), \Alex(U^l x_k)).$
\end{itemize}
Then $(C,\partial, \mathcal{F})$ satisfies the all conditions for being 
an abstract infinity complex.
\end{proof}

On the other hand, in general, an abstract infinity complex does not satisfy the condition (7) in the definition of formal knot complex.
For instance, $\Lambda$ with grading shifted by $2n$ is an abstract infinity complex, but
it is not $\z$-filtered homotopy equivalent to the original $\Lambda$ with respect to either Alexander or algebraic filtration.

\subsubsection{Basic properties}
Here, we discuss several basic properties of formal knot complexes. 
We first consider a change of filtered basis.
\begin{lem}
\label{filtered basis}
Let $C$ be a formal knot complex and $\{x_k\}_{1 \leq k\leq r}$
a filtered basis for $C$.
\begin{enumerate}
\item For any $l \in \z$ and $a \in \{1, \ldots, r\}$, 
the set $\{ x_k\}^{1 \leq k \leq r}_{k \neq a} \cup \{U^l x_a\}$
is also a filtered basis for $C$.
\item For $a,b \in \{1, \ldots, r\}$ with $a \neq b$,
if $\gr(x_a)=\gr(x_b)$, $\Alex(x_a) \geq \Alex(x_b)$ and $\Alg(x_a) \geq \Alg(x_b)$,
then the set $\{ x_k\}^{1 \leq k \leq r}_{k \neq a} \cup \{x_a+x_b\}$
is also a filtered basis for $C$.
Moreover, $\Alex(x_a+x_b)=\Alex(x_a)$ and $\Alg(x_a+x_b)=\Alg(x_a)$.
\end{enumerate}
\end{lem}

\begin{proof}
It is obvious that both $\{ x_k\}^{1 \leq k \leq r}_{k \neq a} \cup \{U^l x_a\}$
and $\{ x_k\}^{1 \leq k \leq r}_{k \neq a} \cup \{x_a+x_b\}$
are free bases for $C$ as a $\Lambda$-module.
Therefore, the first assertion follows from $U^l x_a \in C_{\gr(x_a)-2l}$, $\Alex(U^lx_a)= \Alex(x_a)-l$
and $\Alg(U^l x_a)=\Alg(x_a)-l$.

We consider the second assertion.
Since $x_a+x_b \in C_{\gr(x_a)}=C_{\gr(x_b)}$, the element $x_a + x_b$ is homogeneous.
Next, let $j_a := \Alex(x_a)$, and then $x_a+x_b$ lies in $\falex{j_a}$.
Here we claim that $x_a+x_b \notin \falex{j_a-1}$.
Assume that $x_a+x_b \in \falex{j_a-1}$. Then
$U^{j_a-1}(x_a+x_b)=U^{j_a-1}x_a+ U^{j_a-1}x_b \in \falex{0}$,
and we have a linear combination
$$
U^{j_a-1}x_a+ U^{j_a-1}x_b = \sum_{1 \leq k \leq r} p_k(U) U^{\Alex(x_k)}x_k
$$
where $p_k(U) \in \F[U]$.
However,  the minimal degree of $p_a(U)U^{\Alex(x_a)}=p_a(U)U^{j_a}$ 
is at least $j_a$,
and hence we have $U^{j_a-1} \neq p_a(U)U^{j_a}$.
This contradicts the fact that $\{x_k\}_{1\leq k\leq r}$
is a free basis for $C$ as a $\Lambda$-module.
Therefore, we have $x_a+x_b \notin \falex{j_a-1}$ and 
$\Alex(x_a+x_b)=j_a$.
Now, it is easy to check that
$\{ U^{\Alex(x_k)}x_k\}^{1 \leq k \leq r}_{k \neq a} \cup \{U^{\Alex(x_a+x_b)}(x_a+x_b)\}$
is a free basis for $\falex{0}$ as a $\F[U]$-module.
Similarly, we can check that
$\{ U^{\Alg(x_k)}x_k\}^{1 \leq k \leq r}_{k \neq a} \cup \{U^{\Alg(x_a+x_b)}(x_a+x_b)\}$
is a free basis for $\falg{0}$ as a $\F[U]$-module.
\end{proof}

Next we consider the rank of formal knot complexes.

\begin{lem}
\label{odd}
For any formal knot complex $C$,
the rank of $C$ as a $\Lambda$-module is odd.
\end{lem}

\begin{proof}
Since there exists a chain homotopy equivalence map from $C$
to $\Lambda$ such that the map, its inverse and all chain homotopies are
graded and filtered with respect to the Maslov grading and the algebraic filtration,
we have $H_*(\falg{0}/\falg{-1})= H_0(\falg{0}/\falg{-1}) \cong \F$. In particular, the Euler characteristic of 
$\falg{0}/\falg{-1}$ is 1.
Here, as a $\F$-vector space,
$\{ U^{\Alg(x_k)}x_k\}_{1\leq k\leq r}$ is a basis for $\falg{0}/\falg{-1}$,
and hence $r$ is odd. This completes the proof.
\end{proof}

Finally, by using a fixed filtered basis $\{x_k\}_{1 \leq k \leq r}$, we consider a decomposition
$C= \bigoplus_{(i,j) \in \z^2}C_{(i,j)}$ as a $\F$-vector space, where $C_{(i,j)}$
is defined by 
$$
C_{(i,j)} := \spanF\{U^l x_k \mid (\Alg(U^l x_k),  \Alex(U^l x_k)) = (i,j) \}.
$$
We call it {\it the decomposition of $C$ induced by $\{x_k\}_{1 \leq k \leq r}$}.

\begin{lem}
\label{subcpx decomp}
For any $R \in \CR(\z^2)$, the equality
$$
C_R = \bigoplus_{(i,j) \in R} C_{(i,j)}
$$
holds.
\end{lem}

\begin{proof}
By the definitions of $C_R$ and filtered basis, we see that
$$
C_R = \sum_{(i,j) \in R} (\falg{i} \cap \falex{j})
$$
and
$$
\falg{i} \cap \falex{j} = \spanFU\{ U^{\max\{\Alg(x_k)-i, \Alex(x_k) -j \}} x_k\}_{1\leq k\leq r}.
$$
Therefore, if $(i,j) \in R$ and $U^l x_k \in C_{(i,j)}$, then 
$$
l = \Alg(x_k) - i = \Alex(x_k) -j = \max \{\Alg(x_k)-i, \Alex(x_k) -j \},
$$
and hence $U^l x_k \in \falg{i} \cap \falex{j} \subset C_R$.
This implies $C_R \supset \bigoplus_{(i,j) \in R}C_{(i,j)}$.

Conversely, if $(i,j) \in R$ and $l \geq \max\{\Alg(x_k)-i, \Alex(x_k) -j \}$, then 
$$
U^l x_k \in C_{(\Alg(x_k)-l, \Alex(x_k)-l)}
$$
and
$$
(\Alg(x_k)-l, \Alex(x_k)-l) \leq (i,j).
$$
This implies $U^l x_k \in \bigoplus_{(i,j) \in R}C_{(i,j)}$, and hence 
$C_R \subset \bigoplus_{(i,j) \in R}C_{(i,j)}$.
\end{proof}
As a corollary, we have the following useful lemma.
\begin{lem}
\label{add}
For any $R,R' \in \CR(\z^2)$, we have $C_{R \cup R'} = C_R + C_{R'}$.
\end{lem}
\begin{proof}
By Lemma~\ref{subcpx decomp}, we see that
$$
C_{R \cup R'} = \bigoplus_{(i,j) \in R \cup R'}C_{(i,j)}
= (\bigoplus_{(i,j) \in R}C_{(i,j)}) + (\bigoplus_{(i,j) \in R'}C_{(i,j)})
= C_{R} + C_{R'}.
$$
\end{proof}

\subsection{Commutative monoid structure}
In this subsection, we check that the tensor product of formal knot complexes is also
a formal knot complex.

Let $\mKf$ be the set of the $\z^2$- filtered homotopy equivalence classes of formal knot complexes.
\begin{prop}
\label{tensor}
For any two formal knot complexes $C$ and $C'$, the tuple
$$
\begin{array}{l}
(C \ot{\Lambda} C', \partial \otimes 1 + 1 \otimes \partial, 
\{ \spanF p(\bigcup_{m \in \z} C_m \times C'_{n-m}) \},  \\
\{  \spanF p(\falex{0} \times \falex{j} )\}, 
\{ \spanF p( \falg{0} \times \falg{i} )\})
\end{array}
$$
is a formal knot complex, where 
$p: \Lambda^{C \times C'} \twoheadrightarrow \ C \ot{\Lambda} C'$ is the projection.
Moreover, the set $\mKf$
with product
$$
\mKf \times \mKf \to \mKf : ([C], [C']) \mapsto [C\ot{\Lambda} C'] 
$$
is a commutative monoid.
\end{prop}
\begin{remark}
Note that $p(\falex{j_1} \times \falex{j_2})=
p(\falex{j'_1} \times \falex{j'_2})$ if $j_1+j_2=j'_1+j'_2$,
and hence the definition of the Alexander (resp.\ algebraic) filtration is symmetric.
\end{remark}
\begin{proof}
The fact that $(C \ot{\Lambda} C', \partial \otimes 1 + 1 \otimes \partial)$ is 
a chain complex follows from
ordinary arguments in homological algebra.
Let $\{ x_k\}_{1 \leq k \leq r}$ (resp.\ $\{ x'_l\}_{1 \leq l \leq s}$) be a filtered basis 
 for $C$ (resp.\ $C'$).
Then 
$\{ x_k \otimes x'_l  |^{1 \leq k \leq r}_{1 \leq l \leq s}\}$ is a free basis for 
$C \ot{\Lambda} C'$, and
$$
\{ U^n (x_k \otimes x'_l) \mid 1 \leq k \leq r, 1 \leq l \leq s, n \in \z \}
$$
is a basis for $C \ot{\Lambda} C'$ as a $\F$-vector space.
In particular, the subspace
$$
(C \ot{\Lambda} C')_n := \spanF p(\bigcup_{m \in \z} C_m \times C'_{n-m})
$$ 
is generated by
$
\{
U^{\frac{\gr(x_k)+\gr(x'_l)-n}{2}} (x_k \otimes x'_l) 
\}_{(k,l) \in [n]},
$
where $[n]$ is a subset  of $\{1, \ldots, r\} \times \{1, \ldots, s\}$
such that $(k,l) \in [n]$ if and only if $\gr(x_k)+\gr(x'_l) \equiv n$ (mod 2). This implies that  
$C \ot{\Lambda} C' = \bigoplus_{n \in \z} 
(C \ot{\Lambda} C')_n$ as a $\F$-vector space,
 $\partial((C \ot{\Lambda} C')_n) \subset (C \ot{\Lambda} C')_{n-1}$
and $U((C \ot{\Lambda} C')_n) \subset (C \ot{\Lambda} C')_{n-2}$.
Therefore, the first condition and a part of the fourth and fifth conditions hold.

Next, 
it is obvious that 
$
\{ \spanF p(\falex{0} \times \falex{j} )\}_{j \in \z}
$
gives an increasing sequence of subcomplexes, and
we see that 
$\{ U^{\Alex(x_k)+\Alex(x'_l)-j} (x_k \otimes x'_l)|^{1 \leq k \leq r}_{1 \leq l \leq s }\}$
is a free basis for 
$ \spanF p(\falex{0} \times \falex{j} )$ as a $\F[U]$-module.
Hence the second condition and a part of the fourth and fifth conditions hold.
Similarly, we can verify that the third condition and
the remaining part of the fourth and fifth conditions hold. 

Next we consider the seventh condition.
Here we note that it is easy to check that for the trivial case 
(i.e.\ the case of $C=C'=\Lambda$),
the seventh condition holds. Indeed, the canonical identification
$\Lambda \ot{\Lambda} \Lambda \cong \Lambda$
and its inverse are graded and $\z$-filtered chain isomorphisms (with respect to both filtrations).

Let $f_{\Alex}$ (resp.\ $f'_{\Alex}$) be a $\z$-filtered homotopy equivalence map from $C$ (resp.\ $C'$)
to $\Lambda$ with respect to the Alexander filtration.
Then 
the composition of $f_{\Alex} \otimes f'_{\Alex} : C \ot{\Lambda} C' 
\to \Lambda \ot{\Lambda} \Lambda$ with the canonical identification
$\Lambda \ot{\Lambda} \Lambda \cong \Lambda$ 
is a $\z$-filtered homotopy equivalence map 
with respect to
the grading $\{(C \ot{\Lambda} C')_n\}_{n \in \z}$ and
the filtration $\{ \spanF p(\falex{0} \times \falex{j} ) \}_{j \in \z}$.
Therefore, the seventh condition holds with respect to the Alexander filtration. 
In the same way, we can also prove the seventh condition with respect to the algebraic filtration,
and verify that
$
C \ot{\Lambda} \Lambda \simeq C
$,
$
C \ot{\Lambda} C' \simeq C' \ot{\Lambda} C
$, and if $C \simeq C''$ then
$
C \ot{\Lambda} C' \simeq C'' \ot{\Lambda} C'
$.

Now, to prove the proposition, it suffices to prove the sixth condition, and this follows from taking $\iota \otimes \iota'$, where $\iota: C \to C$
(resp.\ $\iota': C' \to C'$) is a map satisfying the sixth condition for $C$ (resp.\ $C'$).
This completes the proof.
\end{proof}
Now, let $\mK$ be the monoid of the isotopy classes of knots.
Then we see that the connected sum formula of $CFK^{\infty}$ gives a 
monoid homomorphism
$\mK \to \mKf$.
\begin{thm}[\text{\cite[Theorem~7.1]{OS04knot}}]
\label{monoid hom}
The map $\mK \to \mKf: [K] \mapsto [C^K]$ is a monoid homomorphism.
Equivalently, the equality $[C^{K\#J}]=[C^K \ot{\Lambda} C^J]$ holds.
\end{thm}

\subsection{The dual of a formal knot complex}
In this subsection, we check that the dual of a formal knot complex
is also a formal knot complex.

Let $C$ be a formal knot complex. Since $C$ is freely generated by 
a filtered basis $\{ x_k \}_{1\leq k \leq r}$ as a $\Lambda$-module,
the dual $C^* := \Hom_{\Lambda}(C,\Lambda)$ 
is freely generated by the dual basis $\{ x^*_k \}_{1\leq k \leq r}$.
We use the dual basis to define the Maslov grading and two filtrations on $C^*$.

Here we note that $C^*$ is a $\F$-vector space and 
$\{ U^l x^*_k \mid l \in \z, 1 \leq k \leq r \}$ is a basis for $C^*$ as a $\F$-vector space.
Hence we can define a $\F$-linear isomorphism $\Phi : C \to C^*$ 
by $\Phi(U^l x_k) = U^{-l} x_k^*$ . (Remark that since $C$ is infinite-dimensional 
$\F$-vector space, $C^*$ is not isomorphic to $\Hom_{\F}(C, \F)$.) We call $\Phi$
 \textit{the dual isomorphism induced by} $\{ x_k\}_{1 \leq k \leq r}$.

Next, 
let $C/\falex{j}$ (resp.\ $C/\falg{i}$) denote the subspace of $C$ (as a $\F$-vector space)
generated by $\{ U^{l} x_k\}^{l \leq \Alex(x_k) -j -1}_{1 \leq k \leq r }$
(resp.\ $\{ U^{l} x_k\}^{l \leq \Alg(x_k) -i -1}_{1 \leq k \leq r }$).
Then we have 
$$
\begin{array}{ll}
C= \falex{j} \oplus (C/ \falex{j}) &(\text{resp.\ }C= \falg{i} \oplus (C/ \falg{i})),\\
(C/ \falex{j+1}) \subset (C/ \falex{j})
& (\text{resp.\ }(C/ \falg{i+1}) \subset (C/ \falg{i})), \text{ and}\\
U(C/ \falex{j}) = (C/ \falex{j-1})
& (\text{resp.\ } U(C/ \falg{i}) = (C/ \falg{i-1})).
\end{array}
$$
In particular, we see that $\Phi(C/ \falex{j})$ 
(resp.\ $\Phi(C/ \falg{i})$) is a free $\F[U]$-module 
generated by $\{U^{-\Alex(x_k)+j+1} x^*_k \}_{1 \leq k \leq r}$
(resp.\ $\{U^{-\Alg(x_k)+i+1} x^*_k \}_{1 \leq k \leq r}$).
Now, the formal knot complex structure of $C^*$ is described as follows.
\begin{prop}
\label{dual}
Let $\partial^* : C^* \to C^*$ denote the dual of the differential $\partial$ on $C$. 
Then, the tuple
$$
(C^*, \partial^*, \Phi(C_{-n}), 
\Phi(C / \falex{-j-1}), \Phi( C / \falg{-i-1}) )
$$
is a formal knot complex. Moreover, 
for any formal knot complexes $C_1, C_2$, 
if $C_1 \simeq C_2$ then $C_1^* \simeq C_2^*$.
\end{prop}
We call the formal knot complex $C^*$ 
\textit{the dual of $C$}.
Before proving Proposition \ref{dual},
we prove the following lemmas.
Here, $\varepsilon : \Lambda \to \F$ is a $\F$-linear map
defined by $\varepsilon (p(U)) = p(0)$ for each $p(U) \in \Lambda$
(i.e. $\varepsilon$ maps a Laurent polynomial to its constant term).
\begin{lem}
\label{lem dual}
We have the equalities
$$
\begin{array}{l}
\Phi(C_{n})= \left\{ \varphi \in C^* \mid \varepsilon \circ \varphi (\bigoplus_{m \neq n} C_m) = \{ 0\} \right\},\\
\Phi(C / \falex{j})= \left\{ \varphi \in C^* \mid \varepsilon \circ 
\varphi (\falex{j}) = \{ 0\} \right\}, \text{ and}\\
\Phi(C / \falg{i})= \left\{ \varphi \in C^* \mid \varepsilon \circ 
\varphi (\falg{i}) = \{ 0\} \right\}.
\end{array}
$$
In particular,
the subspaces $\Phi(C_{-n})$, $\Phi(C/\falex{-j-1})$ and $\Phi(C/\falg{-i-1})$
are independent of $\Phi$. (We often denote them by $C^*_n$, $\falex{j}(C^*)$ 
and $\falex{i}(C^*)$ respectively.)
\end{lem}

\begin{proof}
We first note that  
$\Phi(C_{n})$ is generated by $\{U^{-\frac{\gr(x_k)-n}{2}}x^*_k\}_{k \in [n]}$.
Now, Suppose that $\varphi$ is in $\Phi(C_{n})$, and then
we have a $\F$-linear combination
$$
\varphi = \sum_{k \in [n]} a_k U^{-\frac{\gr(x_k)-n}{2}}x^*_k.
$$
Thus, for any element $x = \sum_{1\leq k \leq r} p_k(U) x_k \in \bigoplus_{m \neq n} C_m$,
we have
$$
\varphi(x) = \sum_{ k \in [n]} a_k U^{-\frac{\gr(x_k)-n}{2}}p_k(U).
$$
Here, since $x$ is in $\bigoplus_{m \neq n} C_m$, the coefficient of 
$U^{\frac{\gr(x_k)-n}{2}}$ in $p_k(U)$ is zero. 
This implies that 
$$
\varepsilon \circ \varphi(x) = \sum_{k \in [n]}
a_k \varepsilon(U^{-\frac{\gr(x_k)-n}{2}}p_k(U)) = 0.
$$
Conversely, suppose that 
$\varphi = \sum_{1\leq k \leq r} q_k(U) x^*_k \in C^*$ 
satisfies $\varepsilon \circ \varphi (\bigoplus_{m \neq -n} C_m) = \{0\}$.
Here we note that 
the coefficient of $U^l$ in $q_k(U)$ is zero
if and only if
$\varepsilon \circ \varphi (U^{-l} x_k) = 0$.
In addition, for any $k \in [n]$, $U^{-l}x_k$ is in $\bigoplus_{k \neq -n} C_k$
if and only if $l \neq -\frac{\gr(x_k)+n}{2}$,
and hence we have $q_k(U) = a_k U^{-\frac{\gr(x_k)-n}{2}}$ for some $a_k \in \F$.
Otherwise, $U^l x_k \in \bigoplus_{k \neq -n} C_k$ for any $l$,
and hence $q_k(U)=0$.
As a consequence, we have $\varphi = \sum_{k \in [n]} a_k U^{-\frac{\gr(x_k)-n}{2}}x^*_k$.
In a similar way, we can also prove the assertions for $\Phi(C/\falex{-j-1})$
and $\Phi(C/\falg{-i-1})$. 
\end{proof}

\begin{lem}
\label{lem dual 2}
Let $C,C'$ be formal knot complexes 
and $f: C \to C'$ be a $\Lambda$-linear map.
Define a map $f^*: C'^* \to C^*$ by $\varphi \mapsto \varphi \circ f$.
\begin{enumerate}
\item
Fix $k \in \z$. If $f(C_n) \subset C'_{n+k}$ for any $n$, then $f^* (\Phi(C'_{n+k})) \subset \Phi(C_{n})$.
\item
If $f(\falex{j}(C)) \subset \falex{j}(C')$, then 
$f^* (\Phi(C'/\falex{j})) \subset \Phi(C/\falex{j})$.
\item
If $f(\falg{i}(C)) \subset \falg{i}(C')$, then 
$f^* (\Phi(C'/\falg{i})) \subset \Phi(C/\falg{i})$.
\end{enumerate}
\end{lem}

\begin{proof}
Lemma \ref{lem dual} implies that for any $\varphi \in \Phi(C'_{n+k})$, the equalities
$$
\varepsilon \circ (f^* \varphi) (\bigoplus_{m \neq n}C_m) 
= \varepsilon \circ \varphi (f(\bigoplus_{m \neq n}C_m)) 
\subset \varphi(\bigoplus_{m \neq n+k}C'_m) = \{ 0\}
$$
hold, and hence $f^* \varphi \in \Phi (C_n)$.
Similarly, we can prove the second and third assertions in Lemma \ref{lem dual 2}.
\end{proof}

\def\proofname{Proof of Proposition \ref{dual}}
\begin{proof}
The first, second, third and forth conditions immediately follow from
the arguments above Proposition \ref{dual}, the above two lemmas and the equality $U \Phi = \Phi U^{-1}$.
So we first consider the fifth condition. We prove that 
$\{ x^*_k\}_{1\leq k \leq r}$ is a filtered basis.
First, $x^*_k$ is in $\Phi(C_{-\gr(x_k)})$ and hence it is homogeneous. 
Next, it is easy to see that  
$\{U^{-\Alex(x_k)} x^*_k \}_{1 \leq k \leq r}$
is a free basis for $\Phi(C/\falex{-1})$ as a $\F[U]$-module,
and $x^*_k \in \Phi(C/\falex{-j-1})$ if and only if $j = - \Alex(x_k)$. 
These imply that $\{x_k^*\}_{1 \leq k \leq r}$
satisfies the fifth condition with respect to the Alexander filtration.
Similarly, we can also prove that $\{ x_k^*\}_{1 \leq k \leq r}$ satisfies
the condition with respect to the algebraic filtration. Thus, the fifth condition holds.

Next, we consider the seventh condition.
Let $f_{\Alex}: C \to \Lambda$ be a $\z$-filtered homotopy equivalence map 
with respect to the Alexander filtration, and  $g_{\Alex}$ the inverse of $f_{\Alex}$.
Then the dual $g^*_{\Alex} : C^* \to \Lambda^*$ is a chain homotopy equivalence map over 
$\Lambda$, and
Lemma \ref{lem dual 2} implies that the duals of $f_{\Alex}$, $g_{\Alex}$ and all chain homotopies
are  graded  with respect to the pair 
$$
\left(\{ \Phi(C_{-n})\}_{n \in \z},  
\{\Phi(\Lambda_{-n})\}_{n\in \z}\right),
$$
 and $\z$-filtered with respect to the pair
$$
(\{\Phi(C/\falex{-i-1})\}_{i \in \z}, \{\Phi(\Lambda/\falex{-i-1})\}).
$$
Moreover, if we define a $\Lambda$-linear map $\Psi: \Lambda \to \Lambda^*$ 
by $\Psi(1) = 1^*$, then
$\Psi$ is a chain isomorphism satisfying
$$
\begin{array}{lll}
\Psi(\Lambda_n) &=& 
\left\{
\begin{array}{ll}
\{0, U^{-n/2} \cdot1^* \} & (n: \text{ even})\\
0 & (n: \text{ odd})
\end{array}
\right.\\
\ &=& 
\left\{
\begin{array}{ll}
\{0, \Phi(U^{n/2})\} & (n: \text{ even})\\
0 & (n: \text{ odd})
\end{array}
\right.\\
\ &=&
\Phi(\Lambda_{-n})
\end{array}
$$
and
$$
\begin{array}{lll}
\Psi(\falex{i}(\Lambda)) &=& \spanF\{ U^{l} \cdot 1^* \mid l \geq -i \}\\
\ &=& \spanF\{ \Phi(U^{l}) \mid l \leq i  \} \\
\ &=& \Phi(C/\falex{-i-1}).
\end{array}
$$
These imply that
$\Psi$ and the inverse $\Psi^{-1}$ are graded with respect to the pair
$$
\left(\{ \Lambda_{n}\}_{n\in \z},  
\{\Phi(\Lambda_{-n})\}_{n\in \z}\right),
$$
 and $\z$-filtered with respect to the pair
$$
(\{\falex{i}\}_{i \in \z}, \{\Phi(\Lambda/\falex{-i-1})\}).
$$
As a consequence, the composition 
$\Psi^{-1} \circ g^*: C^* \to \Lambda$
satisfies the seventh condition with respect to the Alexander filtration.
In the same way, we can prove the seventh condition with respect to 
the algebraic filtration. In addition, the sixth condition also follows from 
similar arguments.

Finally, we consider the last assertion in Proposition \ref{dual}.
Suppose that
$C_1, C_2$ are formal knot complexes and $f: C_1 \to C_2$ is
a $\z^2$-filtered homotopy equivalence map.
Then Lemma \ref{lem dual 2}
implies that
the dual $f^*: C^*_1 \to C^*_2$ is a
$\z^2$-filtered homotopy equivalence map.
This completes the proof.
\end{proof}
\def\proofname{Proof}
For  knot complexes, the dual complex corresponds to the mirror. 
(Note that the knot Floer homology $\widehat{HFK}$ is treated in 
\cite[Proposition~3.7]{OS04knot}, while the same proof can be applied to 
$CFK^{\infty}$.)
\begin{thm}[\text{\cite[Proposition~3.7]{OS04knot}}]
\label{dual thm}
For any knot $K$, 
the equality $[C^{K^*}]=[(C^K)^*]$ holds.
\end{thm}
In particular, by combining the above theorem with Theorem~\ref{inverse},
we have 
$$[C^{-K^*}]=[(C^K)^*].$$
This fact is important in terms of knot concordance.
About dual complexes,  we give three more lemmas. 
\begin{lem}
\label{dual epsilon}
Let $C$ be a formal knot complex.
Then the $\F$-linear map $\varepsilon_n:C^*_{-n} \to \Hom_{\F}(C_n,\F)$
defined by $\varphi \mapsto \varepsilon \circ \varphi$
is a cochain isomorphism (where we see $\{C^*_{-n}\}_{n \in \z}$ as a graded cochain complex over $\F$).
In particular, we have $\F$-linear isomorphisms
$$
H_{-n}(C^*) \cong H^n(C_*;\F) \cong \Hom_{\F}(H_n(C_*),\F),
$$
where the first isomorphism is the isomorphism induced from $\varepsilon_n$.
\end{lem}

\begin{proof}
The equalities $\partial^* (\varepsilon_n \varphi)
=\varepsilon \circ \varphi \circ \partial = \varepsilon_{n+1}(\partial^*\varphi)$
show that $\{ \varepsilon_n\}_{n \in \z}$ is a cochain map.
We prove that $\varepsilon_n$ is a $\F$-linear isomorphism.
Let $\{x_k\}_{1 \leq k \leq r}$ be
a filtered basis for $C$ and
$\Phi$ the dual isomorphism induced by $\{x_k\}_{1 \leq k \leq r}$.
Then we see that $\{ \varepsilon \circ (U^{-\frac{\gr(x_k)-n}{2}}x^*_k) \}_{k \in [n]}$
coinsides with the dual basis for $\{ U^{\frac{\gr(x_k)-n}{2}} x_k\}_{k \in [n]}$.  Here we note that
$\{ U^{-\frac{\gr(x_k)-n}{2}}x^*_k \}_{k \in [n]}$
is a basis for $C^*_{-n}$, and hence $\varepsilon_n$ is an isomorphism.
\end{proof}

\begin{lem}
\label{dual dual}
For any formal knot complex $C$, 
the $\Lambda$-linear map $\Xi: C \to C^{**}$
defined by $\Xi (x) (\varphi) = \varphi(x)$
($x \in C$, $\varphi \in C^*$)
is a $\z^2$-filtered chain isomorphism.
In particular, $C^{**} \simeq C$.
\end{lem}

\begin{proof}
It is easy to check that $\Xi$ is a chain isomorphism over $\Lambda$.
Moreover, for a fixed filtered basis $\{x_k\}_{1 \leq k \leq r}$ for $C$,
let $\Phi : C \to C^*$ (resp.\ $\Phi^*: C^* \to C^{**}$) be
the dual isomorphism induced by $\{x_k\}_{1\leq k \leq r}$
(resp.\ $\{x^*_k\}_{1 \leq k \leq r}$), 
and then $\Phi^* \circ \Phi = \Xi$.
Hence we have
$$
\begin{array}{l}
\Xi(C_n) = \Phi^*(\Phi(C_n)) = \Phi^*(C^*_{-n}) = C^{**}_n,\\
\Xi(\falex{j}(C))= \Phi^*(\Phi(\falex{j}(C))) = \Phi^*(C^*/\falex{-j-1}) = \falex{j}(C^{**}),
\text{ and}\\
\Xi(\falg{i}(C))= \Phi^*(\Phi(\falg{i}(C))) = \Phi^*(C^*/\falg{-i-1}) = \falg{i}(C^{**}).
\end{array}
$$
These complete the proof.
\end{proof}

\begin{lem}
\label{dual tensor}
For any two formal knot complexes $C$ and $C'$,
the $\Lambda$-linear map 
$\Gamma: C^* \otimes C'^* \to (C \otimes C')^*$
defined by $\Gamma(\varphi \otimes \psi) (x \otimes y) = \varphi(x) \psi(y)$
$(\varphi \in C^*, \psi \in C'^*, x \in C, y \in C')$
is a $\z^2$-filtered chain isomorphism.
In particular, $(C \otimes C')^* \simeq C^* \otimes C'^*$.
\end{lem}
\begin{proof}
The proof is similar to Lemma~\ref{dual dual}.
\end{proof}
\subsection{Stabilizers}

Let $(A,\partial)$ be a chain complex over $\Lambda$.
We call a tuple $$(A, \partial, \{ A_n\}_{n \in \z}, \{\falex{j}\}_{j \in \z} , 
\{\falg{i}\}_{i \in \z})$$
a \textit{stabilizer} if it satisfies the first to sixth conditions in the
definition of formal knot complex and the following:
\begin{condition}
There exists a chain homotopy $\Phi_{\Alex}$ (resp.\ $\Phi_{\Alg}$) on $C$ connecting the identity and the zero-map
which is $\z$-filtered with respect to the Alexander filtration (resp.\ the algebraic filtration).
\end{condition}
\begin{remark} 
The above condition does not imply $A \simeq 0$.
The relation $A \simeq 0$ is corresponding to the existence of chain homotopies $\Phi_{\Alex}$ and $\Phi_{\Alg}$
satisfying the above condition and $\Phi_{\Alex}=\Phi_{\Alg}$.
\end{remark}
Let $C$ (resp.\ $C'$) be a chain complex over $\Lambda$ satisfying the first to sixth conditions
for being a formal knot complex and 
$\{x_k\}_{1 \leq k \leq r}$ (resp.\ $\{x'_l \}_{1 \leq l \leq s}$)
a filtered basis for $C$ (resp.\ $C'$). Then the tuple
$$
(C \oplus C', \partial \oplus \partial', \{ C_n \oplus C'_n \}_{n \in \z}, 
\{\falex{j}(C) \oplus \falex{j}(C') \}_{j \in \z} , 
\{\falg{i}(C) \oplus \falg{i}(C') \}_{i \in \z})
$$
also satisfies the first to sixth conditions for being a formal knot complex,
where $\{(x_k,0)\}_{1 \leq k \leq r} \cup \{(0,x'_l)\}_{1 \leq l \leq s}$
is a filtered basis for the tuple. 
We abbreviate the tuple to $C \oplus C'$. 
\begin{lem}
\label{acyclic}
Let $A$ be a chain complex over $\Lambda$
satisfying the first to sixth conditions for
being a formal knot complex.
Then $A$ is a stabilizer if and only if $H_*(\falex{0})=H_*(\falg{0})=0$.
\end{lem}

\begin{proof}
The only-if-part obviously holds. We prove the if-part.
Suppose that $H_*(\falex{0})=H_*(\falg{0})=0$.
Then, since $U:\falex{0}\to \falex{-1}$ is a chain isomorphism, we have
$H_*(\falex{-1})=0$ and $H_*(\falex{0}/\falex{-1})=0$.
Let $\{x_k\}_{1 \leq k \leq r}$ be a filtered basis for $A$.
By Lemma \ref{filtered basis},
we may assume that $\Alex(x_k)=0$ for any $k$.
Then we see $\falex{0}/\falex{-1}= \spanF\{p x_k\}_{1 \leq k \leq r}$,
where $p: \falex{0} \to \falex{0}/\falex{-1}$ is the projection. 
Moreover, it follows from $H_*(\falex{0}/\falex{-1})=0$ that $r$ is even and there exists a subset
$\{k_1, k_2, \ldots, k_{r/2}\}$ of $\{1, \ldots, r\}$
such that 
$$\spanF\{p x_{k_1}, \ldots, p x_{k_{r/2}}, 
\partial (p x_{k_1}), \ldots, \partial (p x_{k_{r/2}})\}= \falg{0}/\falg{-1}.
$$
This implies that $\Alex(\partial x_{k_i})=0$ for any $1 \leq i \leq r/2$
and
$$
\spanL\{x_{k_1}, \ldots, x_{k_{r/2}}, 
\partial x_{k_1}, \ldots, \partial x_{k_{r/2}}\}= A.$$
Now, define a $\Lambda$-linear map $\Phi_{\Alex}:A \to A$ by
$x_{k_i} \mapsto 0$ and $\partial x_{k_i} \mapsto x_{k_i}$.
Then, it is not hard to check that $\Phi_{\Alex}(C_n) \subset C_{n+1}$,
$\Phi_{\Alex}(\falex{i}) \subset \falex{i}$, and 
$\Phi_{\Alex} \circ \partial + \partial \circ \Phi_{\Alex}$ is equal to the identity on $A$.
This proves the condition for being a stabilizer with respect to the Alexander filtration.
In the same way, we can prove
the condition for being a stabilizer with respect to the algebraic filtration.
\end{proof}
In addition, we can easily check that the following lemmas hold.
\begin{lem}
\label{stabilizer sum}
For two stabilizers $A$ and $A'$, 
the direct sum $A \oplus A'$ is also a stabilizer.
Moreover, for a formal knot complex $C$,
the direct sum $C \oplus A$ is also a formal knot complex.
\end{lem}

\begin{lem}
\label{stabilizer tensor}
For two stabilizers $A$ and $A'$, 
and a formal knot complex $C$,
the tensor products 
$A \otimes_{\Lambda} A'$ and $C \otimes_{\Lambda} A$ are also stabilizers.
\end{lem}

\begin{lem}
\label{stabilizer dual}
For a stabilizer $A$, the dual $A^*$ is also a stabilizer.
\end{lem}

\subsection{$\nu^+$-invariant}

For any formal knot complex $C$, we have
$$H_n(C) \cong H_n(\Lambda) \cong \left\{
\begin{array}{ll}
\F &(n: \text{ even})\\
0 & (\text{otherwise})
\end{array}
\right..
$$
In particular, $H_0(C) \cong \F$.
A cycle $x \in C$ is called a \textit{homological generator}
if $x$ is homogeneous with $\gr(x)=0$ and the homology class $[x] \in H_0(C)$
is non-zero. We define the {\it $\nu^+$-invariant} of $C$ by
$$
\nu^+(C):=\min\{ m \in \z_{\geq 0} \mid C_{\{i\leq0, j\leq m\}} \text{ contains a homological generator}  \}.
$$

\begin{remark}
The above definition of $\nu^+$ is originally that of $\nu^-$.
However, these invariants are the same, and hence  we may define $\nu^+$ as above. 
\end{remark}

Note that the equality 
$$
\nu^+(C)=\min\{ m \in \z_{\geq 0} \mid 
i_{*,0} : H_0(
C_{\{i\leq0, j\leq m\}} )
\to H_0(C) \text{ is surjective}
\}
$$
holds, and hence the value $\nu^+(C)$ is invariant under $\z^2$ -filtered homotopy equivalence.
\begin{prop}
\label{sub additivity}
$\nu^+(C \ot{\Lambda} C') \leq \nu^+(C) + \nu^+(C')$.
\end{prop}

\begin{proof}
Note that 
$C_{\{i\leq0, j\leq m\}}=\falg{0} \cap \falex{m}$,
and hence there exists a homological generator 
$x \in C$ (resp.\ $x' \in C'$) lying in $\falg{0} \cap \falex{\nu^+(C)}$
(resp.\ $\falg{0} \cap \falex{\nu^+(C')}$).
This implies that $x\otimes x' \in C \ot{\Lambda}C'$ is
lying in 
$$
\begin{array}{ll}
p(\falg{0} \times \falg{0}) \cap p(\falex{\nu^+(C)} 
\times \falex{\nu^+(C')}) &
\subset \falg{0} \cap \falex{\nu^+(C)+\nu^+(C')}\\
\ &=(C\ot{\Lambda}C')_{\{i \leq 0, j \leq \nu^+(C)+\nu^+(C')\}}.
\end{array}
$$
Moreover, it is easily seen that $x \otimes x'$ is a homogeneous cycle with 
$\gr(x \otimes x')=0$ (and so $[x \otimes x'] \in H_0(C \ot{\Lambda}C')$), and
the K\"unneth formula 
$
H_*(C) \ot{\Lambda} H_*(C') \hookrightarrow  H_*(C \ot{\Lambda}C') 
$
implies that $[x \otimes x']$ is non-zero. 
Therefore, $x \otimes x'$ is a homological generator, and this completes the proof.
\end{proof}
It is easy to see that the value of $\np$ is unchanged under stabilization.
\begin{lem}
\label{nuplus stabilize}
For any formal knot complex $C$ and stabilizer $A$,
we have $\nu^+(C \oplus A) = \nu^+(C)$. 
\end{lem}
Moreover, $\np$ also has the following property.
\begin{lem}
\label{nuplus dual}
For any formal knot complex $C$, we have
$$
\nu^+(C \otimes_{\Lambda} C^*)=0.
$$
\end{lem}

\begin{proof}
Let $\{x_k\}_{1 \leq k \leq r}$ be a 
filtered basis for $C$.
Then, the element $x = \sum_{1 \leq k \leq r} x_k \otimes x^*_k$ is lying in 
$(C \otimes_{\Lambda}C^*)_{\{i \leq 0, j \leq 0\}}$
and homogeneous with $\gr(x)=0$.
We prove that this $x$ is a homological generator.

Let $( a_{lk})_{1 \leq l,k \leq r} $ be the matrix 
of $\partial : C \to C$ with respect to $\{x_k\}_{1 \leq k \leq r}$,
i.e.\ $\partial x_k = \sum_{1 \leq l \leq r} a_{lk} x_l$.
Then its transpose $( a_{kl})_{1 \leq l,k \leq r} $ is the matrix 
of $\partial^* : C \to C$ with respect to $\{x^*_k\}_{1 \leq k \leq r}$,
and we have
$$
(\partial \otimes 1 + 1 \otimes \partial^*)(x_k \otimes x_k^*)
= \sum_{1 \leq l \leq r} a_{lk} x_l \otimes x_k^*
+ \sum_{1 \leq l \leq r} a_{kl} x_k \otimes x_l^*.
$$
This implies that 
$$
\big( \text{the coefficient of }x_l \otimes x_k^* \text{ in } 
(\partial \otimes 1 + 1 \otimes \partial^*)(x) \big)= 2 a_{lk} = 0
$$
for any $1 \leq l,k \leq r$.
Hence $x$ is a cycle.

Next, we prove that the homology class of $x$ is non-zero.
It is obvious that $\sum_{1 \leq k \leq r} x_k^* \otimes x_k \in C^* \otimes_{\Lambda} C$
is also a cycle.
Here, by using the chain isomorphisms $\Xi$ and $\Gamma$ in
Lemmas \ref{dual dual} and \ref{dual tensor},
we can identify $C^* \otimes C$ with $(C \otimes C^*)^*$
by $(\varphi \otimes y)( z \otimes \psi) = \varphi(z) \psi(y)$
($y,z \in C, \varphi, \psi \in C^*$).
(In other words, $\sum_{1 \leq k \leq r} x^*_k \otimes x_k$ can be seen as a cocycle.)
Now, it follows from Lemma \ref{odd} that $r$ is odd, and hence we have
$$
\big(\sum_{1 \leq k \leq r}x^*_k \otimes x_k \big)(x) 
= \sum_{1 \leq k \leq r} \big(x^*_k (x_k) \big)^2 = r =1 \in \Lambda.
$$
This implies that the homology class of $x$ is non-zero.
\end{proof}

The following proposition is originally proved by Hom \cite{Ho17} in the case of knot complexes.
\begin{prop}[\text{\cite[Proposition 3.11]{Ho17}}]
\label{stable1}
For a formal knot complex $C$,
the equalities $\nu^+(C)=\nu^+(C^*)=0$ holds
if and only if we have the $\z^2$-filtered homotopy equivalence
$$
C \simeq \Lambda \oplus A,
$$
where $A$ is a stabilizer.
\end{prop}
The proof  in \cite{Ho17} is naturally generalized to the case of formal knot complexes.
To prove Proposition~\ref{stable1}, we use the following lemma.
\begin{lem}
\label{stb lem1}
The inequality $\nu^+(C^*)\leq m$ holds
if and only if the projection
$
p_{*,0}: H_0(C) \to H_0(C/C_{\{i\leq -1 \text{ or } j \leq -m-1\}})
$
is injective.
\end{lem}

\begin{proof}
Let $\{ x_k\}_{1 \leq k \leq r}$ be a filtered basis for $C$
and $\Phi$ denote the dual isomorphism induced by $\{ x_k\}_{1 \leq k \leq r}$.
We first prove the only-if-part.

Suppose that $\nu^+(C^*)\leq m$.
Then there exists a homological generator $\varphi \in C^*_0$
lying in 
$$
\begin{array}{ll}
\falg{0}(C^*) \cap \falex{m}(C^*) &= \Phi(C/\falg{-1}) \cap \Phi(C/\falex{-m-1})\\
\ &= \spanFU\{U^{\max\{-\Alg(x_k),-\Alex(x_k)-m\}}x^*_k\mid 1 \leq k \leq r\}.
\end{array}
$$
In particular, we have
$
\varepsilon \circ \varphi(C_{\{i \leq -1 \text{ or } j \leq -m-1\}}) =
\varepsilon \circ \varphi(\falex{-1} + \falg{-m-1} ) =0,
$
and $\varepsilon \circ \varphi$ is
decomposed as $\varepsilon \circ \varphi = \widetilde{\varphi} \circ p$
where $\widetilde{\varphi} \in \Hom_{\F}(C_{\{i \leq -1 \text{ or } j \leq -m-1\}}, \F)$ 
is a cocycle and $p : C \to C/C_{\{i \leq -1 \text{ or } j \leq -m-1\}}$ 
is the projection.
Now, let $x \in C_0$ be a homological generator.
Then we have $\widetilde{\varphi}(p (x))= (\varepsilon \circ \varphi)(x)=1$.
This implies that the homology class 
$[p (x)] \in H_0(C/C_{\{i \leq -1 \text{ or } j \leq -m-1\}})$
is non-zero, and hence 
$p_{*,0}$ is injective.

Conversely, suppose that 
$p_{*,0}$ is injective. Let $x \in C_0$ be a homological generator,
and then we have $p_{*,0}([x]) \neq 0$.
In addition, $\dim_{\F}(C/C_{\{i \leq -1 \text{ or } j \leq -m-1\}})_0$
is finite, and hence we can take a finite $\F$-basis for 
$H_0(C/C_{\{i \leq -1 \text{ or } j \leq -m-1\}})$ containing $p_{*,0}([x])$.
Thus, by using the identification 
$$
\Hom_{\F}(H_0(C/C_{\{i \leq -1 \text{ or } j \leq -m-1\}}),\F) \cong H^0(C/C_{\{i \leq -1 \text{ or } j \leq -m-1\}} ;\F),
$$
 we can take a cocycle 
$\psi \in \Hom_{\F}((C/C_{\{i \leq -1 \text{ or } j \leq -m-1\}})_0,\F)$
whose cohomology class is the dual $(p_{*,0}([x]))^*$.
Moreover, the map 
$\varepsilon_0$ in Lemma \ref{dual epsilon} is bijective,
and hence we can take the inverse $\varphi := \varepsilon^{-1}_0(\psi \circ p) \in C^*_0$.
Note that since $\varepsilon \circ \varphi(x) = \psi(p(x))=1$,
the element $\varphi \in C^*_0$ is a homological generator.
Moreover, the equalities
$$
\varepsilon \circ \varphi (\falg{-1}+ \falex{-m-1}) = \psi \circ p(\falg{-1}+ \falex{-m-1})
= \{0\}
$$
hold,  and hence $\varphi$ lies in $\Phi(C/\falg{-1}) \cap \Phi(C/\falex{-m-1})= 
C^*_{\{i\leq0, j \leq m\}}$.
This proves that $\nu^+(C^*) \leq m$.
\end{proof}

\def\proofname{Proof of Proposition \ref{stable1}}
\begin{proof}
The if-part immediately follows from Lemma \ref{nuplus stabilize}.
To prove the only-if-part, 
we will prove that if $\nu^+(C)=\nu^+(C^*)=0$, then there exists 
a filtered basis $\{x_k\}_{1 \leq k \leq r}$
such that $C$ is decomposed into $\spanL\{x_1\} \oplus \spanL \{x_k\}_{2 \leq k \leq r}$
as a chain complex.
In the situation, the restriction of $\partial$ on $\spanL\{x_1\}$ is zero-map,
and hence it follows from Lemma \ref{acyclic} that $\spanL\{x_1\}$ is a formal knot complex with 
$\spanL\{x_1\} \simeq \Lambda$ and $\spanL \{x_k\}_{2 \leq k \leq r}$
is a stabilizer.

Suppose that $\nu^+(C)=\nu^+(C^*)=0$, and let $\{x_k\}_{1 \leq k \leq r}$ be 
a filtered basis for $C$.
By Lemma \ref{filtered basis},
we may assume that $\gr(x_k)=0$ for $k \in \{1 ,\ldots, r_0\}$ and $\gr(x_k)=1$ for
$k \in \{r_0+1 , \ldots r \}$. Set $r_1 := r-r_0$ and $y_l := x_{r_0+l}$ ($1 \leq l \leq r_1$).
Then, by the definition of $\nu^+$ and Lemma~\ref{stb lem1}, there exists a homological generator 
$x= \sum_{1 \leq k \leq r_0} a_k x_k \in C_0$ such that
$x \in C_{\{i \leq0, j \leq0\}}$, and the homology class of $p(x)$ is non-zero,
where $p: C \to C/C_{\{i \leq -1 \text{ or } j \leq -1 \}}$ is the projection.
This implies that 
\begin{itemize}
\item
if $a_k \neq 0$, then $x_k \in C_{\{i \leq0, j \leq0\}}$, and
\item
there exists a number $k \in \{1, \ldots, r_0\}$ with $a_k \neq 0$ and  
$ x_k \notin C_{\{i \leq -1 \text{ or } j \leq -1\}}$.
\end{itemize}
As a consequence, we have $k' \in \{ 1, \ldots, r_0\}$ such that $a_{k'} \neq 0$ 
and $\Alg(x_{k'}) = \Alex(x_{k'}) =0$.
Moreover, since the inequalities
$$\Alg(x_k) \leq 0 = \Alg( x_{k'})$$
and
$$
\Alex(x_k) \leq 0 = \Alex(x_{k'})
$$
hold for any $k \in \{1, \ldots, r_0\}$ with $a_k \neq 0$,
it follows from Lemma \ref{filtered basis}
that $\{x\} \cup \{ x_k\}^{1\leq k \leq r_0}_{k \neq k'} \cup \{y_l\}_{1\leq l \leq r_1}$
is a filtered basis. We reorder $\{ x_k\}^{1\leq k \leq r_0}_{k \neq k'}$
as $\{ x_k\}_{2\leq k \leq r_0}$.

Next, we will change $\{ x_k\}_{2\leq k \leq r_0}$ into $\{ x'_k\}_{2\leq k \leq r_0}$ so that
$\{x\} \cup \{ x'_k\}_{2\leq k \leq r_0} \cup \{y_l\}_{1\leq l \leq r_1}$ 
is still a filtered basis and
$\partial (\{y_l\}_{1 \leq l \leq r_1}) \subset \spanF \{x'_k\}_{2 \leq k \leq r_0}$.
Then, we can conclude that
both $\spanL\{x\}$ and 
$\spanL(\{ x'_k\}_{2\leq k \leq r_0} \cup \{y_l\}_{1\leq l \leq r_1})$
are subcomplexes, and this will complete the proof.
To obtain such $\{x'_k\}$, we first note that
$$
\{p x\} \cup \{ p x_k \mid 2 \leq k \leq r_0 \text{ and } \Alg(x_k),\Alex(x_k) \geq 0 \}
$$
is a basis for $p(C_0)$.
We reorder $\{x_k\}_{2 \leq k \leq r_0}$ so that
$\{p x\} \cup \{ p x_k \}_{2 \leq k \leq r'_0}$
is a basis for $p(C_0)$. (Here $r'_0 := \dim_{\F}p(C_0)$.)
Let $(a_{kl})^{1 \leq k\leq r'_0}_{1 \leq l \leq r_1}$
be the matrix of $p \circ (\partial|_{C_1}):C_1 \to p(C_0)$ with respect to
the pair $(\{y_l\}_{1\leq l \leq r_1},  \{ p x_k\}_{2\leq k \leq r'_0} \cup \{ p x\})$,
i.e.\ $p \circ \partial (y_l) = \sum_{1 \leq k \leq r'-1}a_{kl}p x_{k+1} + a_{r' l} p x$.
Then we can replace $\{y_l\}_{1 \leq l \leq r_1}$ with 
a basis $\{y'_l\}_{1 \leq l \leq r_1}$ 
so that the corresponding matrix 
$(a'_{kl})^{1 \leq k\leq r'_0}_{1 \leq l \leq r_1}$
is in reduced column echelon form.
Here, since $[p x] \neq 0$ in $H_0(C/ C_{\{i \leq -1 \text{ or } j \leq -1\}})$,
$p x$ is not contained in $p \circ \partial (C_1)$ and
the last row of $(a'_{kl})^{1 \leq k\leq r'_0}_{1 \leq l \leq r_1}$
does not contain any leading coefficient.
In particular, if $a'_{r'_0l}\neq0$,
then there exists a number $k_l$ in $\{1, \ldots, r'_0-1\}$
such the $k_l$-th row contains the $l$-th leading coefficient.
(Namely, $a'_{k_ll'}=\delta_{ll'}$, where $\delta_{ll'}$ is the Kronecker delta.)
Now, we define a set $\{x'_k \}_{2 \leq k \leq r_0}$ by
$$
x'_k =
\left\{
\begin{array}{ll}
x_k+x & (\text{if $k= k_l-1$ for some $l \in \{1, \ldots, r_1\}$ with $a'_{r'_0l}\neq0$})\\
x_k & (\text{otherwise})
\end{array}
\right..
$$
Then, it follows from Lemma \ref{filtered basis} that
$\{x\} \cup \{ x'_k\}_{2\leq k \leq r_0} \cup \{y_l\}_{1\leq l \leq r_1}$ 
is a filtered basis.
Moreover,
the replacement of $\{p x_k\}_{2 \leq k \leq r'_0}$ 
with $\{p x'_k\}_{2 \leq k \leq r'_0}$ changes
$(a'_{kl})^{1 \leq k\leq r'_0}_{1 \leq l \leq r_1}$
so that the last row is zero vector. 
This implies that
$$
\begin{array}{ll}
p \circ \partial(\{y_l\}_{1 \leq l \leq r_1}) 
&\subset
p \circ \partial(\spanF \{y_l\}_{1 \leq l \leq r_1}) \\
\ &=
p \circ \partial(\spanF \{y'_l\}_{1 \leq l \leq r_1}) \subset
\spanF \{ p x'_{k} \}_{2 \leq k \leq r'_0},
\end{array}
$$
and hence we have
$$
\partial({y_l}_{1 \leq l \leq r_1})
\subset p^{-1}(\spanF \{ p x'_{k} \}_{2 \leq k \leq r'_0})
= \spanF\{ x'_k\}_{2 \leq k \leq r_0}.
$$
This completes the proof.
\end{proof}
\def\proofname{Proof}

\begin{cor}
\label{cor stable}
Let $C$ and $C'$ be formal knot complexes.
If $\nu^+(C)=\nu^+(C^*)=0$,
then $\nu^+(C' \otimes_{\Lambda} C) = \nu^+(C')$.
\end{cor}

\begin{proof}
By Proposition \ref{stable1},
we have $C \simeq \Lambda \oplus A$.
Here, Lemma \ref{stabilizer tensor} says
that $C' \otimes_{\Lambda} A$ is a stabilizer,
and  it is easy to show that 
$C' \otimes_{\Lambda} (\Lambda \oplus A)
\simeq C' \oplus (C' \otimes_{\Lambda} A)$. 
Therefore, by Lemma \ref{nuplus stabilize},
we have
$$
\nu^+ (C' \otimes_{\Lambda} C) =
\nu^+(C' \oplus (C' \otimes_{\Lambda} A))=
 \nu^+(C').
$$
\end{proof}
Here we refer to the following theorem of Hom and Wu,
which is one of the most important facts 
for obtaining concordance invariants from $CFK^\infty$.
\begin{thm}[\cite{HW16}]
\label{nuplus and g_4}
For a knot $K$, the inequality $\nu^+(C^K) \leq g_4(K)$ holds.
In particular, if $K$ is a slice knot, then 
$\nu^+(C^K) = \nu^+((C^K)^*) = 0$.
\end{thm}

\subsection{$\nu^+$-equivalence}

Two elements $[C],[C'] \in \mKf$ are 
\textit{ $\nu^+$-equivalent} (and denoted $[C] \nuplus [C']$
or $C \nuplus C'$)
if $\nu^+(C \otimes_{\Lambda} C'^*) = \nu^+(C^* \otimes_{\Lambda} C')=0$.
Note that by Propositions \ref{tensor} and \ref{dual},
the values
$\nu^+(C \otimes_{\Lambda} C'^*)$ and $\nu^+(C^* \otimes_{\Lambda} C')$
are independent of the choice of representatives.
\begin{prop}
\label{nuplus equiv}
The relation $\nuplus$ is an equivalence relation on
$\mKf$.
\end{prop}

\begin{proof}
The reflexivity (i.e.\ $[C]\nuplus[C]$) follows from Lemma \ref{nuplus dual}.
The symmetry ($[C] \nuplus [C']$ if and only if $[C'] \nuplus [C]$)
directly follows from the definition.
We prove the transitivity.
Suppose that $[C_1] \nuplus [C_2]$ and $[C_2] \nuplus [C_3]$.
Then, Proposition \ref{sub additivity}, 
Lemma \ref{nuplus dual} and Corollary \ref{cor stable}
imply
$$
\begin{array}{lll}
\nu^+(C_1 \otimes_{\Lambda} C_3^*) &=&
\nu^+ \big(
(C_1 \otimes_{\Lambda} C_3^* ) \otimes_{\Lambda} 
(C_2 \otimes_{\Lambda} C_2^*) \big) \\
\ &=&
\nu^+ \big(
(C_1 \otimes_{\Lambda} C_2^* ) \otimes_{\Lambda} 
(C_2 \otimes_{\Lambda} C_3^*) \big)\\
\ &\leq&
\nu^+(C_1 \otimes_{\Lambda} C_2^* ) + 
\nu^+(C_2 \otimes_{\Lambda} C_3^*)=0.
\end{array}
$$
Similarly, we can prove that
$\nu^+(C_1^* \otimes C_3)=0$ holds.
\end{proof}
We call the equivalence class of a formal knot complex $C$ under $\nuplus$
\textit{the $\nu^+$-equivalence class} or \textit{$\nu^+$-class of $C$}, and denote
it by $[C]_{\nu^+}$.
Then, we can see that Hom's stable homotopy theorem in \cite{Ho17} is naturally generalized to
formal knot complexes.
\begin{thm}[\cite{Ho17}]
\label{stable2}
Two formal knot complexes $C$ and $C'$ are 
$\nu^+$-equivalent
if and only if we have the $\z^2$-filtered homotopy equivalence
$$
C \oplus A \simeq C' \oplus A',
$$
where $A,A'$ are stabilizers.
\end{thm}

\begin{proof}
It follows from Lemma \ref{nuplus stabilize} and Proposition \ref{stable1}
that $C \nuplus C'$ 
if and only if 
$C \otimes_{\Lambda} C'^* \simeq \Lambda \oplus A$
where $A$ is a stabilizer.
Thus, if $C \nuplus C'$, then there exist
stabilizers $A_1, A_2$ so that $C^* \ot{\Lambda} C' \simeq \Lambda \oplus A_1$
and $C \ot{\Lambda} C^* \simeq \Lambda \oplus A_2$, and we have
$$
\begin{array}{ll}
C \oplus (C \ot{\Lambda} A_1)
&\simeq
C \ot{\Lambda} ( \Lambda \oplus A_1)
\simeq
C \ot{\Lambda} (C^* \ot{\Lambda} C')\\
\ &\simeq
(C \ot{\Lambda} C^*) \ot{\Lambda} C'
\simeq
(\Lambda \oplus A_2) \ot{\Lambda} C'
\simeq
C' \oplus (C' \ot{\Lambda} A_2).
\end{array}
$$
Conversely,
if 
$
C \oplus A \simeq C' \oplus A',
$
then there exists a stabilizer $A''$ so that
$$
\begin{array}{ll}
(C \ot{\Lambda} C'^* )
\oplus (A \ot{\Lambda} C'^*)
&\simeq
(C \oplus A )
\ot{\Lambda} C'^*
\simeq
(C' \oplus A') \ot{\Lambda} C'^*
\\
\ &\simeq
(C^* \ot{\Lambda} C'^* )
\oplus (A' \ot{\Lambda} C'^*)\\
\ &\simeq
\Lambda \oplus A''
\oplus (A' \ot{\Lambda} C'^*),
\end{array}
$$
and hence $\np(C \ot{\Lambda} C'^*) = \np((C \ot{\Lambda} C'^* )
\oplus (A \ot{\Lambda} C'^*)) = 0$.
Similarly, we have $\np(C^* \ot{\Lambda}C')=0$.
\end{proof}

Here, due to Theorem \ref{nuplus and g_4},
the $\nu^+$-class of $C^K$ can be seen 
as a knot concordance invariant of $K$.
\begin{cor}[\cite{Ho17}]
For a knot $K$, $[K]_{\nu^+}:=[C^K]_{\nu^+}$ is
a knot concordance invariant of $K$.
\end{cor}
\begin{proof}
If two knots $K$ and $J$ are concordant, then both $K\#(-J^*)$ and $(-K^*)\#J$
are slice knots. Thus, by Theorem~\ref{nuplus and g_4}, we have
$$
\nu^+(C^K\otimes_{\Lambda}(C^J)^*)
= \nu^+(C^{K\# (-J^*)}) = 0
$$
and
$$
\nu^+((C^K)^*\otimes_{\Lambda}C^J)
= \nu^+(C^{(-K^*)\# J}) = 0.
$$
\end{proof}

\subsection{Formal knot concordance group}
Now, {\it the formal knot concordance group} $\mCf$ is obtained as follows.
\begin{prop}
\label{abelian}
The quotient set $\mCf :=\mKf / \nuplus$
with product
$$
\mCf \times \mCf \to \mCf : ([C]_{\nu^+}, [C']_{\nu^+}) 
\to [C \otimes_{\Lambda} C']_{\nu^+}
$$
is an abelian group.
In particular,
the projection $\mKf \twoheadrightarrow \mCf$
is a monoid homomorphism.
\end{prop}

\begin{proof}
We first verify that the product is well-defined.
Suppose that $[C]_{\nu^+} = [C'']_{\nu^+}$, and then 
$\nu^+(C \otimes_{\Lambda} C''^*) = \nu^+(C^* \otimes_{\Lambda} C'') =0$.
Thus, it follows from Proposition \ref{sub additivity} and
Lemma \ref{nuplus dual} that
$$\nu^+((C \otimes_{\Lambda} C') \otimes_{\Lambda}
(C'' \otimes_{\Lambda} C')^*)
=\nu^+((C \otimes_{\Lambda} C''^*) \otimes_{\Lambda}
(C' \otimes_{\Lambda} C'^*))=0
$$
and
$$\nu^+((C \otimes_{\Lambda} C')^* \otimes_{\Lambda}
(C'' \otimes_{\Lambda} C'))
=\nu^+((C^* \otimes_{\Lambda} C'') \otimes_{\Lambda}
(C' \otimes_{\Lambda} C'^*))=0.
$$
Similarly, we can prove that if $[C']_{\nu^+} = [C'']_{\nu^+}$ then 
$[C \otimes_{\Lambda} C']_{\nu^+} =[C \otimes_{\Lambda} C'']_{\nu^+}$.
Now, the commutativity immediately follows from 
$C \otimes_{\Lambda} C' \simeq C' \otimes_{\Lambda} C$,
and obviously the projection $\mKf \twoheadrightarrow \mCf$ is a monoid homomorphism.
\end{proof}
As a consequence, we have the following theorem, which is stated in
Section~1 as Theorem~\ref{comm diag}.
\begin{thm}
The map $\mC \to \mCf : [K]_c \mapsto [C^K]_{\nu^+}$ is
a well-defined group homomorphism.
As a consequence, we have the following commutative diagram:
$$
\begin{CD}
\mK @>  [K] \mapsto [C^K] >> \mKf \\
@V [K] \mapsto [K]_c VV @VV [C] \mapsto [C]_{\nu^+} V\\
\mC @>> [K]_c \mapsto [C^K]_{\nu^+} > \mCf
\end{CD}
$$
\end{thm}

\subsection{Partial order on $\mCf$}
In this subsection, we introduce a partial order on $\mCf$,
which is a generalization of the partial order on $\mC_{\np}$
defined in \cite{Sa18full}.
Here, as a new observation, we give an interpretation of the $\np$-equivalence
and the partial order on $\mCf$ using quasi-isomorphisms.

For two $\nu^+$ -classes $[C]_{\nu^+},[C']_{\nu^+} \in \mCf$,
we denote $[C]_{\nu^+} \leq [C']_{\nu^+}$
if 
the equality $\nu^+(C\otimes_{\Lambda}C'^*)=0$ holds.

\begin{prop}
The relation $\leq$ is a partial order on $\mCf$.
\end{prop}

\begin{proof}
This immediately follows from Proposition \ref{sub additivity},
Lemma \ref{nuplus dual} 
and the definition of $\nuplus$.
\end{proof}

For two formal knot complexes, a chain map $f: C \to C'$ over $\Lambda$ is
a \textit{$\z^2$-filtered quasi-isomorphism}
if $f$ is $\z^2$-filtered, graded, and induces an isomorphism
$f_*: H_*(C) \to H_*(C')$.
Then, the $\nu^+$-equivalence and the partial order on $\mCf$
can be translated into the words of the existence of 
$\z^2$-filtered quasi-isomorphisms.
\begin{thm}
\label{quasi thm1}
Two formal knot complexes $C$ and $C'$ are $\nu^+$-equivalent
if and only if 
there exist $\z^2$-filtered quasi-isomorphisms
$$f: C \to C' \text{ and } g: C' \to C.$$
\end{thm}

\begin{thm}
\label{quasi thm2}
Two $\nu^+$-classes $[C]_{\nu^+}$ and $[C']_{\nu^+}$ satisfy
$[C]_{\nu^+} \geq [C']_{\nu^+}$ if and only if
there exists a $\z^2$-filtered quasi-isomorphism
 $C \to C'.$
\end{thm}

To prove these theorems, we first prove the following lemma.

\begin{lem}
\label{quasi lem}
Let $C$ and $C'$ be formal knot complexes.
If there exists a $\z^2$-filtered quasi-isomorphism
$f: C \to C'$, then $[C]_{\nu^+} \geq [C']_{\nu^+}$.
\end{lem}

\begin{proof}
Note that under the hypothesis of the lemma, 
$f \otimes id_{C^*}: C \ot{\Lambda} C^* \to C' \ot{\Lambda} C^*$
is also a $\z^2$-filtered quasi-isomorphism.
Moreover, by Lemma \ref{nuplus dual}, we can take a homological generator $x$ of 
$C \ot{\Lambda} C^*$ lying in $(C \ot{\Lambda} C^*)_{\{i\leq 0, j \leq 0\}}$.
Now, we see that $f \otimes id_{C^*}(x)$ is a homological generator 
of $C' \ot{\Lambda} C^*$ lying in $(C' \ot{\Lambda} C^*)_{\{i\leq0, j\leq 0\}}$,
and hence $\nu^+(C' \otimes C^*)=0$.
\end{proof}

\def\proofname{Proof of Theorem~\ref{quasi thm1}}
\begin{proof}
The if-part directly follows from Lemma \ref{quasi lem}.
We prove the only-if-part. Suppose that $C \nuplus C'$.
Then, by Theorem \ref{stable2},
we have a $\z^2$-filtered homotopy equivalence map
$$
f':C \oplus A \overset{\simeq}{\longrightarrow} C' \oplus A',
$$
where $A,A'$ are stabilizers.
Let $i : C \hookrightarrow C \oplus A$ be the inclusion
and $p: C' \oplus A' \twoheadrightarrow C'$ the projection.
Then, all of $i,$ $f'$ and $p$ are $\z^2$-filtered quasi-isomorphisms,
and hence we have the $\z^2$-filtered quasi-isomorphism
$$
f := p \circ f' \circ i: C \to C'.
$$
Similarly, we can construct a $\z^2$-filtered quasi-isomorphism 
$g:C' \to C$.
\end{proof}

\def\proofname{Proof of Theorem~\ref{quasi thm2}}
\begin{proof}
The if-part coincides with Lemma \ref{quasi lem}.
We prove the only-if-part.
Suppose that $[C]_{\nu^+} \geq [C']_{\nu^+}$.
Then the equality
$\nu^+(C' \otimes_{\Lambda} C^*)=0$ holds, and hence 
$(C' \otimes_{\Lambda} C^*)_{\{i\leq 0, j \leq 0\}}$ contains a homological
generator $x$. Hence, if we define a $\Lambda$-linear map
$$
f \colon \Lambda \to C' \otimes_{\Lambda} C^*
$$
so that $f(1) = x$, then $f$ is a $\z^2$-filtered quasi isomorphism.
In addition, the map
$$
f \otimes id_C \colon C  \to C' \otimes_{\Lambda} C^* \otimes_{\Lambda} C
$$
is also a $\z^2$-filtered quasi isomorphism.
Moreover, since $(C' \otimes_{\Lambda} C^* \otimes_{\Lambda} C) \nuplus C'$,
Theorem~\ref{quasi thm1} gives a $\z^2$-filtered quasi isomorphism
$$
C' \otimes_{\Lambda} C^* \otimes_{\Lambda} C \to C'.
$$
By combining these quasi isomorphisms, we obtained the desired quasi-isomorphism.
\end{proof}
\def\proofname{Proof}
When one wants to construct a $\z^2$-filtered quasi-isomorphism concretely,
the following lemma is useful.
\begin{lem}
\label{quasi-isom basis}
Let $C$ and $C'$ be formal knot complexes and $f \colon C \to C'$ be a chain map
over $\Lambda$ such that
\begin{itemize}
\item $f$ maps a homological generator $C$ to that of $C'$, and
\item
for a filtered basis $\{x_k\}_{1 \leq k \leq r}$ of $C$ and any $k$, 
we have $$(\Alg(f x_k),\Alex(f x_k)) \leq (\Alg(x_k),\Alex(x_k)).$$ 
\end{itemize}
Then, $f$ is a $\z^2$-filtered quasi-isomorphism.
\end{lem}
\begin{proof}
Since
$\falex{j}(C) = \spanFU\{ U^{\Alex(x_k) -j} x_k\}_{1\leq k\leq r}$,
 we have
\begin{eqnarray*}
f(\falex{j}(C)) &=& \spanFU\{ U^{\Alex(x_k) -j} f x_k\}_{1\leq k\leq r}\\
&\subset& \spanFU\{ U^{\Alex(f x_k) -j} f x_k\}_{1\leq k\leq r} \subset \falex{j}(C').\\
\end{eqnarray*}
Similarly, we have $f(\falg{i}(C)) \subset \falg{i}(C')$. Now, for any $R \in \CR(\z^2)$,
we see that
\begin{eqnarray*}
f(C_R) &=& f(\sum_{(i,j) \in R}\falg{i}(C)\cap \falex{j}(C))\\
&\subset& \sum_{(i,j) \in R}f(\falg{i}(C))\cap f(\falex{j}(C))\\
&\subset& \sum_{(i,j) \in R}\falg{i}(C')\cap \falex{j}(C') = C'_R.
\end{eqnarray*}
It is easy to see that $f$ is a quasi-isomorphism.
\end{proof}

Set
$\mC_{\nu^+} :=  \image(\mC \to \mCf \colon [K]_c \mapsto [C^K]_{\nu^+})$.
Then $\mC_{\nu^+}$ is naturally identified with a quotient group of $\mC$,
and the partial order on $\mCf$ induces a partial order on $\mC_{\nu^+}$.
We note that the induced partial order coincides with
the order defined in author's paper \cite{Sa18full}.
In particular, Proposition~1.5 in \cite{Sa18full} is naturally generalized to $\mCf$.
\begin{prop}[\text{\cite[Proposition~1.5]{Sa18full}}]
The partial order on $\mCf$ has the following properties:
\begin{enumerate}
\item For elements $x,y,z \in \mCf$, if $x \leq y$, then $x+z \leq y+z$.
\item For elements $x,y \in \mCf$, if $x \leq y$, then $-y \leq -x$.
\end{enumerate}
\end{prop}
On the other hand, 
for the case of $\mC_{\nu^+}$, 
we also have the following geometric estimates.
(Here, {\it full-twist operations} are defined as follows. 
Let $K$ be a knot and $D$ a disk in $S^3$ which intersects $K$ in its interior.
By performing $(-1)$-surgery along $\partial D$, we obtain a new knot $J$ in $S^3$ from $K$.
Let $n = \lk(K, \partial D)$. Since reversing the orientation of $D$ does not affect the result, we may assume that $n\geq0$. Then we say that $K$ is deformed into $J$ by 
{\it a positive full-twist with $n$-linking}, and call such an operation a {\it full-twist operation}.)
\begin{thm}[\text{\cite[Theorem~1.6]{Sa18full}}]
Suppose that a knot $K$ is deformed into a knot $J$ by 
a positive full-twist with $n$-linking.
\begin{enumerate}
\label{full-twist}
\item If $n=0$ or $1$, then $[J]_{\nu^+} \leq [K]_{\nu^+}$. 
\item If $n \geq 3$, then $[J]_{\nu^+} \nleq [K]_{\nu^+}$.
In particular, if the geometric intersection number between $K$ and $D$
is equal to $n$, then $[J]_{\nu^+} > [K]_{\nu^+}$.
\end{enumerate}
\end{thm}

\subsection{Invariants of $\nu^+$-classes}
\label{section invariants}
In this subsection, we review the $V_k$-sequence \cite{NW15},
the $\tau$-invariant \cite{OS03tau}, the $\Upsilon$-invariant \cite{OSS17}
and the $\Upsilon^2$-invariant \cite{KL18}
as invariants of formal knot complexes under $\nu^+$-equivalence.
Here we use $\z^2$-filtered quasi-isomorphisms to prove the invariance of them.
\subsubsection{$V_k$-sequence}
The $V_k$-sequence defined by Ni and Wu \cite{NW15} is a family of 
$\z_{\geq 0}$-valued invariants which is parametrized by $\z_{\geq 0}$.
Concretely, for a formal knot complex $C$ and $k \in \z_{\geq 0}$, 
the value $V_k(C)$ is defined by
$$
V_k(C) = \dim_{\F} \Big(\coker (i_* \colon H_*(C_{\{i \leq 0, j \leq k\}}) \to 
H_*(C_{\{i \leq 0\}}))\Big).
$$
In particular, we have the equality
$$
\nu^+(C) = \min \{ k \in \z_{\geq 0} \mid V_k(C)= 0 \}.
$$
Moreover, we can use homological generators to determine $V_k(C)$.
\begin{lem}
\label{V_k homol gen}
For any $k \in \z_{\geq 0}$, the equality
$$
V_k(C) = \min \{ m \in \z_{\geq 0} \mid C_{\{i\leq m, j \leq k+m\}} 
\text{ {\rm contains a homological generator}} \}
$$
holds.
\end{lem}
\begin{proof}
Denote the value of the right-hand side of the equality in Lemma~\ref{V_k homol gen}
by $V'_k(C)$. We first prove that $V_k(C) \geq V'_k(C)$.
Since $H_*(C_{\{i \leq 0\}}) \cong \F[U]$ and the map
$i_* \colon H_*(C_{\{i \leq 0, j \leq k\}}) \to H_*(C_{\{i \leq 0\}})$
is a $\F[U]$-linear map, if $\image i_{*, 2m}=H_{2m}(C_{\{i \leq 0\}})$
then $\image i_{*, 2n}=H_{2n}(C_{\{i \leq 0\}})$ for any $n \leq m$.
This implies that
$$
i_{*, -2V_k(C)} \colon H_{-2V_k(C)}(C_{\{i \leq 0, j \leq k\}}) 
\to H_{-2V_k(C)}(C_{\{i \leq 0\}})
$$
is surjective. Moreover, the map 
$
i_{*, n} \colon H_{n}(C_{\{i \leq 0\}}) 
\to H_{n}(C)
$
is an isomorphism for any $n \leq 0$. Consequently, we see that
there exists a cycle $x \in C_{-2V_k(C)}$ lying in $C_{\{i \leq 0, j \leq k\}}$
such that the homology class $[x] \in H_{-2V_k(C)}(C)$ is non-zero.
This implies that $U^{-V_k(C)} x \in C_0$ is a homological generator lying 
in $C_{\{i \leq V_k(C), j \leq k+V_k(C)\}}$. 
Therefore, we have $V_k(C) \geq V'_k(C)$.

Conversely, since $C_{\{i \leq V'_k(C), j \leq k+V'_k(C)\}}$ contains a homological generator
$x$,  the cycle $U^{V'_k(C)}x \in C_{-2V'_k(C)}$ is lying in $C_{\{i \leq 0, j \leq k\}}$.
This implies that the map
$$
i_{*, -2V'_k(C)} \colon H_{-2V'_k(C)}(C_{\{i \leq 0, j \leq k\}}) 
\to H_{-2V'_k(C)}(C_{\{i \leq 0\}})
$$
is surjective, and hence $V_k(C) \leq V'_k(C)$.
\end{proof}
Now, we can easily see that $V_k$ is a well-defined map on $\mCf$ and preserve the 
partial order.
\begin{cor}
\label{V_k invariance}
If $[C]_{\nu^+} \leq [C']_{\nu^+}$, then $V_k (C)\leq V_k(C')$ for any $k \geq 0$.
In particular, $V_k$ is a well-defined map $\mCf \to \z_{\geq 0}$.  
\end{cor}

\begin{proof}
Suppose that $[C]_{\nu^+} \leq [C']_{\nu^+}$, and then we have a $\z^2$-filtered
quasi-isomorphism $f \colon C' \to C$. Here, by using Lemma~\ref{V_k homol gen},
we can take a homological generator $x \in C'$ lying in 
$C'_{\{i \leq V_k(C') , j \leq k+V_k(C') \}}$.
Then, $f(x)$ is a homological generator of $C$ lying in 
$C_{\{i \leq V_k(C') , j \leq k+V_k(C') \}}$. This completes the proof.
\end{proof}
In addition, we also have the following properties of $V_k$. 
\begin{cor}
\label{V_k np}
For any $k \in \z_{\geq 0}$, we have
$$
V_k(C)-1 \leq V_{k+1}(C) \leq V_k(C).
$$
In particular, for any $0 \leq k \leq \np(C)$, the inequality $V_k(C)+k \leq \np(C)$ holds.
\end{cor}
\begin{proof}
The first assertion immediately follows from the fact that 
\begin{eqnarray*}
C_{\{i \leq m-1 , j \leq (k+1)+(m-1) \}} &\subset&
C_{\{i \leq m , j \leq k+m \}}\\
&\subset& 
C_{\{i \leq m , j \leq (k+1)+m \}}.
\end{eqnarray*}
Next, for any $0 \leq k \leq \np(C)$, we see that 
$$
V_k(C) \leq V_{k+1}(C)+1 \leq \cdots \leq V_{\np(C)}(C) +(\np(C)-k) = \np(C)-k.
$$
This completes the proof.
\end{proof}
Moreover, we have a connected sum inequality for $V_k$.
(For knot complexes, it is given in \cite{BCG17}.)
\begin{cor}
For any formal knot complexes $C,C'$ and $k,k' \in \z_{\geq 0}$, we have 
$$
V_{k+k'}(C\otimes_{\Lambda} C') \leq V_k(C)+ V_{k'}(C').
$$
\end{cor}

\begin{proof}
By Lemma~\ref{V_k homol gen}, we have a homological generator
$x \in C$ (resp.\ $x' \in C'$) which is lying in 
$C_{\{i \leq V_k(C) , i \leq k+V_k(C) \}}$ 
(resp.\ $C_{\{i \leq V_{k'}(C') , i \leq k+V_{k'}(C') \}}$).
This implies that $x \otimes x'$ is a homological generator of $C\otimes_{\Lambda}C'$
lying in 
$$
(C\otimes_{\Lambda}C')_{\{i \leq V_k(C)+V_{k'}(C') , i \leq (k+k')+V_k(C) +V_{k'}(C') \}}.
$$
This completes the proof.
\end{proof}

For the case of knot complexes, $V_k(K) := V_k(C^K)$ is an important invariant
because it completely determines all correction terms of all positive Dehn surgeries along $K$.
To state the fact precisely, we fix several notations.
For coprime integers $p,q>0$, let $S^3_{p/q}(K)$ denote the $p/q$-surgery along $K$. 
Note that there is a canonical identification between the set
of $\Spin^c$ structures over $S^3_{p/q}(K)$ 
and $\{ i \mid  0 \leq i \leq p-1\}$. This identification can be
made explicit by the procedure in \cite[Section 4, Section 7]{OS11rational}.
Let $d(S^3_{p/q}(K), i)$ denote the correction term of $S^3_{p/q}(K)$
with  the $i$-th $\Spin^c$ structure ($0 \leq i \leq p-1$). 
\begin{prop}[\text{\cite[Proposition~1.6]{NW15}}]
\label{V_k d}
The equality
$$
d(S^3_{p/q}(K),i) = d(S^3_{p/q}(O),i) - 2 
\max \left\{ V_{\lfloor \frac{i}{q} \rfloor}(K), V_{\lfloor \frac{p+q -1-i}{q} \rfloor}(K) \right\}
$$
holds,
where $O$ denotes the unknot and $\lfloor \cdot \rfloor$ is the floor function.
\end{prop}

\subsubsection{$\tau$-invariant}
\label{section tau}
Let $C$ be a formal knot complex.
Define 
$$
\hC:= C_{\{i\leq 0\}} / C_{\{i\leq -1\}} \\
$$
and
$$
\mhF_m := C_{\{i\leq 0, j \leq m\}} / C_{\{i\leq -1, j \leq m\}}
$$
for any $m \in \z$.
Then we see $H_*(\hC)=H_0(\hC) \cong \F$, and $\{ \mhF_m \}_{m \in \z}$ is an increasing sequence of subcomplexes on $\hC$, i.e.\ a $\z$-filtration on $\hC$. 
We call a cycle $x \in \hC$ a
{\it hat-generator}
if $x$ is homogeneous with $\gr(x)=0$ and the homology class $[x]\in H_0(\hC)$
is non-zero. We define the {\it $\tau$-invariant} of $C$ by
$$
\tau(C) := \min\{m \in \z \mid \mhF_m \text{ contains a hat-generator} \}.
$$
We can use homological generators to determine $\tau(C)$ like $V_k(C)$.
\begin{lem}
\label{tau homol gen}
The equality
$$
\tau(C) =\min \{ m \in \z_{\geq 0} \mid C_{\{i\leq -1\}\cup \{i \leq 0, j \leq m\}} 
\text{ {\rm contains a homological generator}} \}
$$
holds.
\end{lem}

\begin{proof}
Denote the value of right-hand side of the equality in Lemma~\ref{tau homol gen}
by $\tau'(C)$. We first prove that $\tau(C) \geq \tau'(C)$.
By the definition of $\tau(C)$, there exists a chain 
$x \in C_{\{i\leq -1\}\cup \{i \leq 0, j \leq \tau(C)\}}$ such that
$p(x) \in \mhF_{\tau(C)}$ is a hat-generator,
where $p \colon C_{\{i \leq 0\}} \to \hC$ is the projection.
Moreover, since the induced map 
$p_{*,0} \colon H_0(C_{\{i \leq 0\}}) \to H_0(\hC)$ is an isomorphism, 
there exists a 0-chain $y \in C_{\{i \leq -1\}}$ such that 
$\partial y = \partial x$. In particular, $x-y$ is a homological generator of $C$
lying in $C_{\{i\leq -1\}\cup \{i \leq 0, j \leq \tau(C)\}}$. (Note that 
$p_{*,0}([x-y])=[p(x-y)]=[p(x)] \neq 0$.)
Therefore, we have $\tau(C) \geq \tau'(C)$.

Conversely, since $C_{\{i\leq -1\}\cup \{i \leq 0, j \leq \tau'(C)\}}$ contains
a homological generator $x'$ and the above map $p_{*,0}$ is an isomorphism,
$p(x')$ is a hat-generator lying in $\mhF_{\tau'(C)}$. This gives $\tau(C) \leq \tau'(C)$.
\end{proof}
Now, by the same arguments as the proof of Corollary~\ref{V_k invariance},
we have the following.
\begin{cor}
\label{tau invariance}
If $[C]_{\nu^+} \leq [C']_{\nu^+}$, then $\tau (C)\leq \tau(C')$.
In particular, $\tau$ is a well-defined map $\mCf \to \z$.  
\end{cor}
In addition, $\tau$ is related to $\np$ as follows.
\begin{cor}
\label{tau np}
The inequality $\tau(C) \leq \np(C)$ holds.
\end{cor}
\begin{proof}
This follows from
$
C_{\{ i \leq 0, j \leq \np(C) \}} \subset C_{\{i \leq -1\} \cup \{ i \leq 0, j \leq \np(C) \}}.
$
\end{proof}

One of the most important property of $\tau$-invariant
is the following additivity.
\begin{prop}
\label{tau hom}
$\tau$ is a group homomorphism as a map $\mCf \to \z$.
\end{prop}

\begin{proof}
Let $C$ and $C'$ be formal knot complexes.
Then we can see from Proposition~\ref{tensor} that
the $\z$-filtered homotopy equivalence
$$
(\widehat{C\otimes_{\Lambda}C'}, \{\mhF_m\}) \simeq (\hC \otimes_{\F} \hC', 
\{\spanF p(\bigcup_{\mu+\mu'=m}\mhF_{\mu}\times \mhF'_{\mu'})\})
$$
holds, where $p \colon \F^{\hC \times \hC'} \twoheadrightarrow \hC\otimes_{\F} \hC'$
is the projection.
Next, let $x \in \mhF_{\tau(C)}$ (resp.\ $x' \in \mhF'_{\tau(C')}$) a hat-generator.
Then, in a similar way to the proof of Proposition~\ref{stable1}, we have the $\z$-filtered 
homotopy-equivalence
$$
\hC \simeq \spanF\{x\} \oplus A \text{ (resp.\  
$\hC' \simeq \spanF\{x'\} \oplus A'$ )},
$$
where $A$ and $A'$ are acyclic $\z$-filtered chain complexes.
Consequently, the  $\z$-filtered homotopy equivalence
$$
\widehat{C\otimes_{\Lambda}C'} \simeq \spanF\{x\otimes x'\} \oplus A''
$$
holds for some acyclic $\z$-filtered chain complex $A''$, and this implies that
$\tau(C\otimes_{\Lambda}C')= \tau(C)+ \tau(C')$.
\end{proof}
As a consequence, we have the original $\tau$-invariant for knots.
\begin{cor}[\cite{OS03tau}]
The map $[K]_c \mapsto \tau(C^K)$ is a group homomorphism as a map 
$\mC \to \z$.
\end{cor}
\subsubsection{$\Upsilon$-invariant}
For any $t \in [0,2]$ and $s \in \R$, the set
$$
R^t(s) := \{(i,j) \in \z^2 \mid (1-\frac{t}{2})i + \frac{t}{2}j \leq s \}
$$
is a closed region. Hence, if we denote $C_{R^t(s)}$ by $\mF^t_s$,
then we have a $\R$-filtration $\{\mF^t_s\}_{s \in \R}$ of $C$.
We define
$$
\upsilon_C(t) := \min\{ s \in \R \mid \mF^t_s \text{ contains a homological generator} \}
$$
and
$$
\Upsilon_C(t) := -2 \upsilon_C(t).
$$
\begin{remark}
This definition of $\Upsilon$ is due to Livingston \cite{Li17} rather than 
the original one \cite{OSS17}. 
\end{remark}
Since there exist finitely many homological generators of $C$ and their
Alexander and algebraic filtrations are finite, $\upsilon_C(t)$ and $\Upsilon_C(t)$
are finite values. In the same way as $V_k$ and $\tau$,
we can prove the following proposition.
\begin{prop}
\label{upsilon invariance}
If $[C]_{\nu^+} \leq [C']_{\nu^+}$, then $\Upsilon_C (t) \geq \Upsilon_{C'}(t)$
for any $t \in [0,2]$.
In particular, $\Upsilon(t) \colon [C]_{\np} \mapsto \Upsilon_{C}(t)$ is a well-defined map 
$\mCf \to \R$ for any $t \in [0,2]$.  
\end{prop}
In addition, we can see $\Upsilon$ as a linear approximation of $V_k$
in the following sense.
\begin{prop}
For any $t \in [0,2]$ and $k \in \z_{\geq 0}$, 
the inequality
$$
\Upsilon_C(t) \geq -kt-2V_k(C)
$$
holds. In particular, $\Upsilon_C(t) \geq -\np(C)t$ holds.
\end{prop}
\begin{proof}
This follows from
$
C_{\{i \leq V_k(C), j \leq k + V_k(C)\}} \subset C_{\{(1- \frac{t}{2})i + \frac{t}{2}j \leq V_k(C) + \frac{t}{2}k \}}.
$
\end{proof}
Moreover, The additivity of $\Upsilon(t)$ is also obtained in the same way as $\tau$. 
\begin{prop}
\label{upsilon hom}
For any $t \in [0,2]$, $\Upsilon(t)$ is a 
group homomorphism as a map $\mCf \to \R$.
\end{prop}

We can generalize the following properties of the original $\Upsilon$-invariant
to formal knot complexes. The proof is similar to \cite[Theorem~8.1]{Li17}.
\begin{prop}
\label{upsilon prop}
For any formal knot complexes $C$, the following properties hold.
\begin{enumerate}
\item The map $\Upsilon_C \colon [0,2] \to \R, t \mapsto \Upsilon_C(t)$
is a continuous piecewise linear function.
\item For any regular point $t$ of $\Upsilon_C$ and filtered basis 
$\{x_k\}_{1 \leq k \leq r}$,
there exists an element $x_l \in \{x_k\}_{1 \leq k \leq r}$ with $\gr(x_l) =0$
such that 
$$
\Upsilon_C(t')= -2 \Alg(x_l) + (\Alg(x_l)-\Alex(x_l))t'
$$
at any point $t'$ nearby $t$.
\item Let $t$ be a singular point of $\Upsilon_C$ and 
$\{x_k\}_{1 \leq k \leq r}$ a filtered basis.
Then there exists two elements $x_l, x_{l'} \in \{x_k\}_{1 \leq k \leq r}$ with 
$\gr(x_l)=\gr(x_{l'})=0$
such that 
\begin{itemize}
\item $\Alex(x_l) - \Alex(x_{l'}) = (1 - \frac{2}{t}) (\Alg(x_l)- \Alg(x_{l'}))$,
\item the equality
$$
\Upsilon_C(t')= -2 \Alg(x_l) + (\Alg(x_l)-\Alex(x_l))t'
$$
holds at any point $t'$ nearby $t$ satisfying $t'< t$, and
\item the equality
$$
\Upsilon_C(t')= -2 \Alg(x_{l'}) + (\Alg(x_{l'})-\Alex(x_{l'}))t'
$$
holds at any point $t'$ nearby $t$ satisfying $t'> t$. 
\end{itemize}
\end{enumerate}
\end{prop}

As a consequence of the above arguments, we have the following corollaries.
Here, $\PL([0,2], \R)$ denotes the set of continuous piecewise linear functions on $[0,2]$.
\begin{cor}
The map $\Upsilon \colon [C]_{\np} \mapsto \Upsilon_C$ is a group homomorphism
$\mCf \to \PL([0,2], \R)$.
\end{cor}

\begin{cor}
The map $[K]_c \mapsto \Upsilon_{C^K}$ is a group homomorphism
$\mC \to \PL([0,2], \R)$.
\end{cor}

Here we mention that the gradient of $\Upsilon_C$ nearby $0$ is  equal to $-\tau(C)$.
The proof is the same as \cite[Theorem 14.1]{Li17}.
\begin{prop}
For any sufficiently small $t>0$, we have 
$\Upsilon_C(t)= -\tau(C) t$.
\end{prop}
\subsubsection{$\Upsilon^2$-invariant}
Let $C$ be a formal knot complex, $\{x_k\}_{1 \leq k \leq r}$ a filtered basis
and $\{C_{(i,j)}\}$ the induced decomposition of $C$.
Define {\it the support of $\{C_{(i,j)}\}$} by
$$
\mathcal{P} :=\{(i,j)\in \z^2 \mid C_{(i,j)} \neq 0 \}.
$$
In addition, consider {\it the support line for} $\mF^t_s$ by
$$
\mathcal{L}^t_s:=\{ (i,j) \in \z^2 \mid (1- \frac{t}{2})i + \frac{t}{2}j =s \}.
$$
Now, for any $t \in [0,2]$, set
$$
\mathcal{P}_{t} :=\mathcal{P} \cap \mathcal{L}^t_{\upsilon_C(t)}.
$$
Then, we see that $\mathcal{P}_{t} \neq \emptyset$ for any $t$.
Moreover, from Proposition~\ref{upsilon prop}, we have the following proposition.
\begin{prop}
\label{pivot}
The following assertions hold:
\begin{enumerate}
\item
For any $t \in [0,2]$ and small $\delta >0$, the
intersection $\mathcal{P}_{t}\cap \mathcal{P}_{t- \delta}$ 
(resp.\ $\mathcal{P}_{t}\cap \mathcal{P}_{t+ \delta}$)
has exactly one point. (We denote these points by $p^-_t$ and $p^+_t$, respectively.)
\item 
The function $\Upsilon_K$ has a singularity at $t$
if and only if $p^-_t \neq p^+_t$.
\end{enumerate}
\end{prop}
In light of this proposition, for small $\delta >0$, we set
$$
\mathcal{Z}^{\pm}_t(C) := \{\text{homological generator in } 
\mF^{t \pm \delta}_{\upsilon_C(t \pm \delta)}  \}.
$$
If $\mathcal{Z}^-_t(C)  \cap \mathcal{Z}^+_t(C) = \emptyset$, 
then for any $s \in [0,2]$, we define
$$
\upsilon^2_{C,t}(s) := 
\min \{r \in \R \mid \exists z^{\pm}\in \mathcal{Z}^{\pm}_t(C), 
[z^-]=[z^+] \text{ in } H_0(\mF^t_{\upsilon_C(t)} + \mF^{s}_{r}) \}.
$$
Now, we can define the {\it $\Upsilon^2$-invariant} of $C$ as 
$$
\Upsilon^2_{C,t}(s) := 
\begin{cases}
-2 (\upsilon^2_{C,t}(s)-\upsilon_C(t)) & \text{if } \mathcal{Z}^-(C)  \cap 
\mathcal{Z}^+(C)  = \emptyset\\
\infty & \text{if } \mathcal{Z}^-(C)   \cap \mathcal{Z}^+(C)  \neq \emptyset
\end{cases}
.
$$
From the view point of $\z^2$-filtered quasi isomorphism, we have 
the following inequality.
\begin{prop}
Fix $t \in (0,2)$. If $[C]_{\np} \leq [C']_{\np}$ and $\Upsilon_C(t')=\Upsilon_{C'}(t')$ 
for any point $t'$ nearby $t$, then $\Upsilon^2_{C,t}(s) \geq \Upsilon^2_{C',t}(s)$
for any $s \in [0,2]$.
\end{prop}
\begin{proof}
Take $\delta>0$ sufficiently small so that Propsition~\ref{pivot}
holds at given $t$ for both $C$ and $C'$.
Let $z'^{\pm} \in \mathcal{Z}^{\pm}_t(C')$ such that 
$$[z'^-]=[z'^+] \text{ in }H_0(\mF^r_{\upsilon_{C'}(t)}(C')
+ \mF^s_{\upsilon^2_{C',t}(s)}(C')),
$$
and $f \colon C' \to C$ a $\z^2$-filtered quasi isomorphism.
Then, since $\upsilon_C(t')= \upsilon_{C'}(t')$
for any point $t'$ nearby $t$, 
we see  $f(z'^{\pm}) \in \mathcal{Z}^{\pm}_t(C)$.
Now, we have the equalities
$$
[f(z'^-)]=f_*([z'^-])=f_*([z'^+])= 
[f(z'^+)]
$$
as elements of $H_0(\mF^r_{\upsilon_C(t)}(C)+ \mF^s_{\upsilon^2_{C',t}(s)}(C))$.
Hence, if $\mathcal{Z}^-_t(C) \cap \mathcal{Z}^+_t(C) = \emptyset$, 
then $C'$ also satisfies $\mathcal{Z}^-_t(C') \cap \mathcal{Z}^+_t(C') = \emptyset$
and
we have the inequality
$$
\upsilon^2_{C,t}(s) \leq \upsilon^2_{C,t}(s),
$$
which gives the desired inequality. Otherwise, $\Upsilon^2_{C,t}(s) = \infty$,
and hence the desired inequality obviously holds.
\end{proof}

As a corollary, we have the invariance of $\Upsilon^2$.
(Note that $\Upsilon^2$ is originally given as an invariant of formal knot complexes
in \cite{KL18}.)
\begin{thm}[\text{\cite[Theorem~4.8]{KL18}}]
For any $t \in (0,2)$ and $s \in [0,2]$, 
$\Upsilon^2_{t}(s) \colon [C]_{\np} \mapsto \Upsilon^2_{C,t}(s)$
is a well-defined map $\mCf \to \R \cup \{\infty\}$.
In particular, $\Upsilon^2_{K,t}(s) := \Upsilon^2_{C^K,t}(s)$
is a knot concordance invariant.
\end{thm}

We also mention the following sub-additivity of $\Upsilon^2_{C,t}(t)$.
\begin{thm}[\text{\cite[Theorem~5.1]{KL18}}]
For any formal knot complexes $C, C'$ and $t \in (0,2)$, we have
$$
\Upsilon^2_{C\otimes_{\Lambda}C', t}(t) 
\geq \min\{\Upsilon^2_{C,t}(t), \Upsilon^2_{C',t}(t)\}.
$$
\end{thm}

\section{Geometric estimates}
In this section, we prove the following theorem.

\setcounter{section}{1}
\setcounter{thm}{4}
\begin{thm}
For any knot $K$, we have
$$
-g_4(K)[T_{2,3}]_{\nu^+} \leq [K]_{\nu^+} \leq g_4(K)[T_{2,3}]_{\nu^+}.
$$
\end{thm}
\setcounter{thm}{0}
\setcounter{section}{3}

To prove the theorem, we consider replacing a given knot $K$ several times.
We start with the following lemma.
\begin{lem}
\label{slice lem}
For any knot $K$,
there exists a knot $K'$ concordant to $K$ which bounds a ribbon surface with genus $g_4(K)$.
\end{lem}

\begin{proof}
Let $F$ be a surface in $B^4 \cong (S^3 \times [0,1]) / (S^3 \times \{ 1 \})$ with genus $g_4(K)$ and $\partial F = K \subset S^3 \times \{0\}$.
Then, a similar argument to \cite[Lemma2.1]{Sa11}
shows that $F$ can be isotoped to a surface $F'$ in $B^4$ such that
the composition $f: F' \hookrightarrow (S^3 \times [0,1]) / (S^3 \times \{ 1 \}) 
\overset{p_2}{\twoheadrightarrow} [0,1]$ is a Morse function, and $f$ satisfies
\begin{enumerate}
\item all births happen at time $\frac{1}{6}$ (we denote the number of births by $b$),
\item $b$ saddles happen at time $\frac{2}{6}$,
\item the time $\frac{3}{6}$ is a regular value and $f^{-1}(\frac{3}{6})$ is connected,
\item the remaining saddles happen at time $\frac{4}{6}$, and
\item all deaths happen at $\frac{5}{6}$.
\end{enumerate}
In particular, we see that $f^{-1}([0,\frac{3}{6}])$ is a (ribbon) concordance from $K$ to
$K':= f^{-1}(\frac{3}{6})$, and $f^{-1}([\frac{3}{6},1])$ is a ribbon surface in 
$(S^3 \times [\frac{3}{6},1]) / (S^3 \times \{ 1 \}) \cong B^4$ whose boundary is $K'$ and genus is $g_4(K)$. This completes the proof.
\end{proof}
Next, by using full-twists, we construct a surface embedded in $S^3$.
\begin{lem}
\label{ribbon lem}
If a knot $K$ bounds a genus $g$ ribbon surface, then there exists a
knot $K'$ with genus $g$
which obtained from $K$ only by adding positive full-twists with $1$-linking.
\end{lem}
\begin{proof}
Suppose that $K$ bounds a genus $g$ ribbon surface $F$ with $n$ ribbon singularities.
Then, for proving the lemma, it suffices to find a positive full-twist with $1$-linking deforming $K$ into a knot $K'$
which bounds a genus $g$ ribbon surface with $n-1$ ribbon singularities.

Let $\Sigma_g$ be an abstract genus $g$ surface with $\partial \Sigma_g \cong S^1$,
and $f: \Sigma_g \to S^3$ an immersion with $f(\Sigma_g)=F$.
Choose a ribbon singularity $b$ on $F$. Then 
$f^{-1}(b)$ consists of two arcs in $\Sigma_g$, one of which is properly embedded and the other is lying in $\Int \Sigma_g$. Denote the arc in $\Int \Sigma_g$ by $\widetilde{b}$, and take an arc $a$
in $\Sigma_g$ 
such that $\Int a$ avoids the preimage of all ribbon singularities on $F$, and
one end of $a$ is in $\partial \widetilde{b}$ and the other is in $\partial \Sigma_g$.
Then $f(a)$ is an arc in $F$ which connects $b$ to $\partial F$, and $\Int f(a)$ avoids all singularities on $F$. Thus, we can take a (small) tubular neighborhood $N$ of $f(a) \cup b$
such that $(N, F \cap N)$ is diffeomorphic to the pair of the 3-ball and the immersed surface
shown in Figure~\ref{local1}.

Now, we take a twisting disk $D$ as shown in the right-hand side of Figure~\ref{local2}.
After adding a positive full-twist along $D$, we have a new ribbon surface $F'$ which coincides
with $F$ in $S^3 \setminus N$, and $(N, F' \cap N)$ is diffeomorphic to the pair of the 3-ball and the embedded surface shown in the right-hand side of Figure~\ref{local2}. By the construction, it is obvious that $K':=\partial F'$
is obtained from $K$ by a positive full-twist with $1$-linking, and 
$F'$ is a genus $g$ ribbon surface with $n-1$ ribbon singularities.
This completes the proof.
\end{proof}

\begin{figure}[htbp]
\begin{center}
\includegraphics[scale= 1.0]{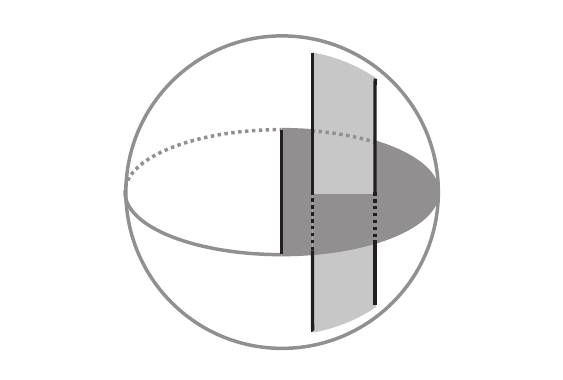}
\caption{A local picture near $f(a) \cup b$}
\label{local1}
\end{center}
\end{figure}
\begin{figure}[htbp]
\begin{minipage}[]{0.45\hsize}
\includegraphics[scale = 1.0]{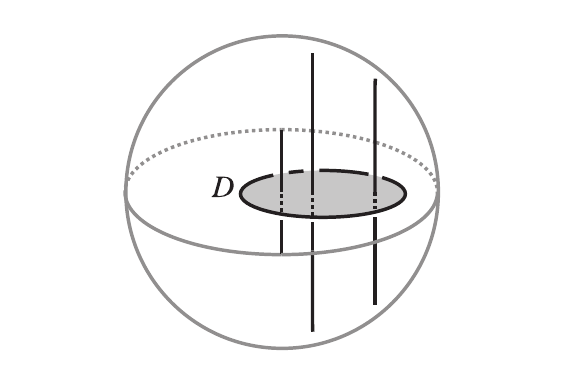}
 \end{minipage}
 \begin{minipage}[]{0.45\hsize}
\begin{center}
\includegraphics[scale = 1.0]{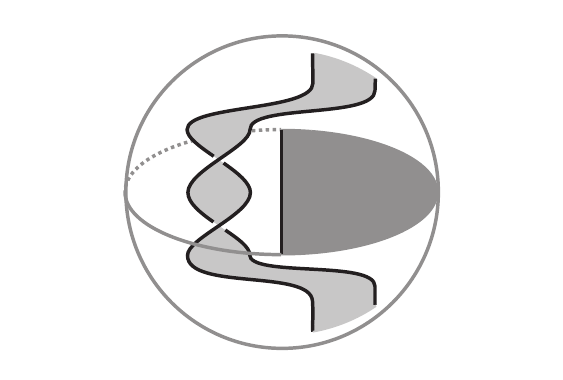}
\end{center}
 \end{minipage}
\caption{A disk $D$ and a surface $F'$}
\label{local2}
\end{figure}

For $m,n \in \Bbb{Z}$, 
let $K_{m,n}$ denote the $(m,n)$-twist knot, whose diagram is shown in Figure~\ref{fig twist knot}.
Then, the last replacement is stated as follows.

\begin{lem}
\label{3-dim lem}
Any genus $g$ knot is deformed into 
$K_{m_1,n_1}\# \cdots \# K_{m_{g}, n_{g}}$ only by adding positive full-twists 
with $0$-linking, where $m_i, n_i \in \Bbb{Z}_{>0}$ (for all $i \in \{1, \ldots, g\}$). 
\end{lem}

\begin{proof}
Let $K$ be a genus $g$ knot
and $F$ a genus $g$ surface with boundary $K$.
By an isotopy, we can assume that $F$ is of the form of Figure~\ref{surface1},
where $L$ is obtained from a string link (with $4g$ strings) by parallelizing the string link with some framings. (The framings are characterized by the choice of $\{m'_1, n'_1, \ldots, m'_g, n'_g\}$.)  
Then, as shown in Figure \ref{pass move}, positive full-twists with 0-linking can realize both directions of pass moves with framings changing, and hence such full-twists can deform $F$ into a surface $F'$ with new framings $\{m_1, n_1, \ldots, m_g, n_g\}$, which is shown in Figure \ref{surface2}. 
Moreover, by 
adding positive full-twists with 0-linking as shown in Figure \ref{framing}, we may assume that
all $m_i, n_i$ are positive.  Here it is obvious that the boundary of $F'$ is 
$K_{m_1,n_1}\# \cdots \# K_{m_{g}, n_{g}}$, and this fact completes the proof.
\end{proof}

\begin{figure}[htbp]
\begin{center}
\includegraphics[scale= 1.0]{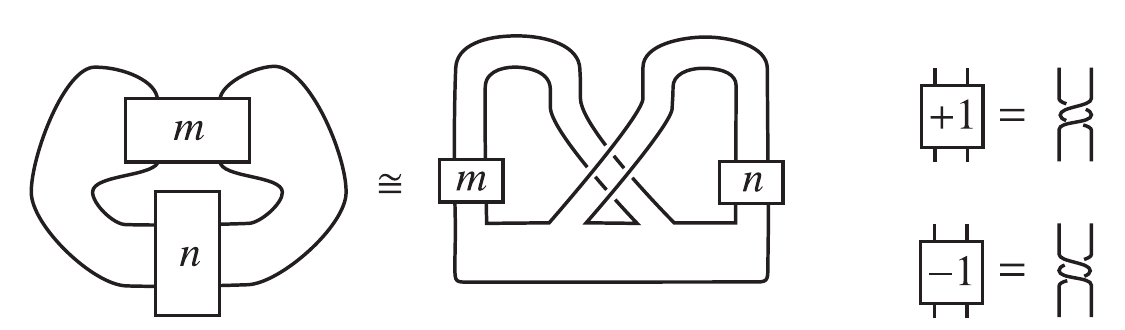}
\caption{The $(m,n)$-twist knot $K_{m,n}$ \label{fig twist knot}}
\end{center}
\end{figure}
\begin{figure}[htbp]
\begin{center}
\includegraphics[scale= 1.0]{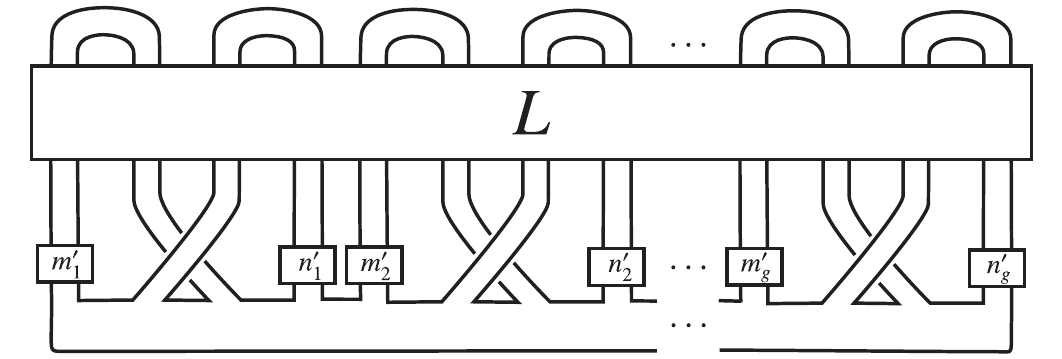}
\caption{A description of $F$ by a string link \label{surface1}}
\end{center}
\end{figure}
\begin{figure}[htbp]
\begin{center}
\includegraphics[scale= 1.0]{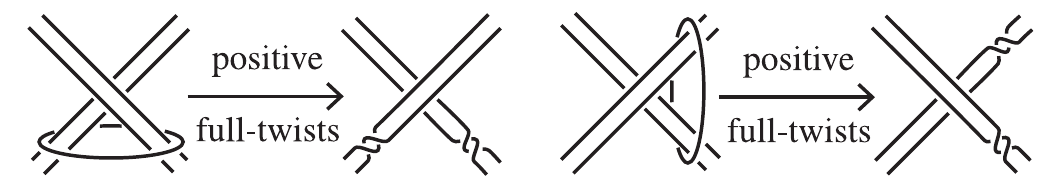}
\caption{Pass moves with framings changing \label{pass move}}
\end{center}
\end{figure}
\begin{figure}[htbp]
\begin{center}
\includegraphics[scale= 1.0]{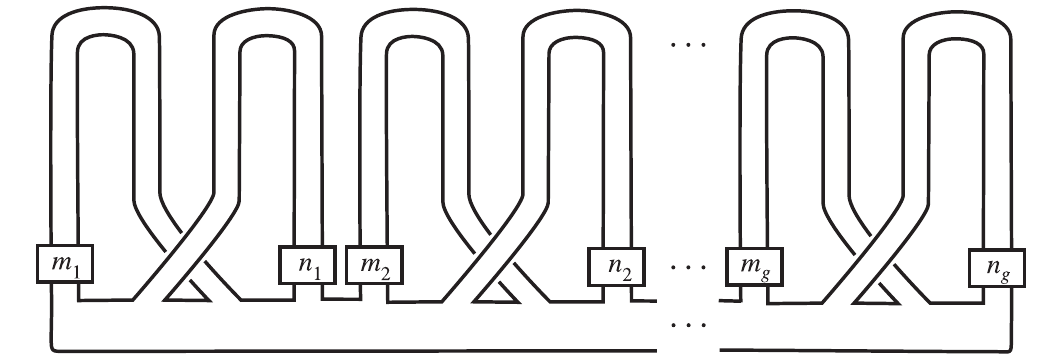}
\caption{A surface $F'$ \label{surface2}}
\end{center}
\end{figure}
\begin{figure}[htbp]
\begin{center}
\includegraphics[scale= 1.3]{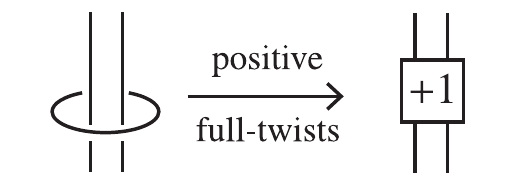}
\caption{A positive full-twist increase a framing \label{framing}}
\end{center}
\end{figure}

Here we note that all $K_{m,n}$ are 2-bridge knots and hence alternating knots.
For alternating knots, the following strong classification theorem of $\nu^+$-classes
follows from \cite[Section~3.1]{Pe13}.
\begin{thm}[\text{\cite[Section~3.1]{Pe13}}]
\label{alt}
For any alternating knot $K$, we have $[K]_{\np} =-\frac{\sigma(K)}{2}[T_{2,3}]_{\np}$,
where $\sigma(K)$ is the knot signature of $K$.
\end{thm}
Now we can determine the $\np$-classes of the $K_{m,n}$.
\begin{lem}
\label{twist knot}
For any $m,n>0$, $[K_{m,n}]_{\nu^+}=-[T_{2,3}]_{\nu^+}$. 
\end{lem}

\begin{proof}
It is easy to verify that for any $m,n>0$, we have $\sigma(K_{m,n})=2$.
Therefore, Theorem~\ref{alt} completes the proof.
\end{proof}
Now we prove Theorem~\ref{geom}.
\def\proofname{Proof of Theorem \ref{geom}}
\begin{proof}
Fix a knot $K$. Then,
Lemma~\ref{slice lem} provides a knot $K'$ such that $[K']_{\np}=[K]_{\np}$ and
$K'$ bounds a ribbon surface with genus $g_4(K)$.
Moreover, it follows from Lemma~\ref{ribbon lem} 
and Lemma~\ref{3-dim lem}
that there exists a sequence of finitely many positive full-twists with 0 or 1-linking
which deforms $K'$ into $K_{m_1,n_1}\# \cdots \# K_{m_{g_4(K)}, n_{g_4(K)}}$
for some $m_i,n_i \in \z_{>0}$ ($i \in \{1, \ldots, g_4(K)\}$).
Therefore, by Theorem~\ref{full-twist} and Lemma~\ref{twist knot}, 
we have
$$
[K]_{\np}=[K']_{\np} \geq \sum_{1\leq i \leq g_4(K)} [K_{m_i,n_i}]_{\np} = -g_4(K)[T_{2,3}]_{\np}.
$$
Since $g_4(-K^*)=g_4(K)$, we also have 
$$
-[K]_{\np}= [-K^*]_{\np} \geq -g_4(K)[T_{2,3}]_{\np}.
$$
This completes the proof.
\end{proof} 
\def\proofname{Proof}

\section{Algebraic estimates}
In this section, we establish several algebraic estimates for the $\np$-classes,
and prove Theorem~\ref{main thm} and Theorem~\ref{discriminant}.
\subsection{Genus of a formal knot complex}
We first define the genus of formal knot complexes.
\subsubsection{Maximal and minimal degrees}
For a formal knot complex $C$, set
$$
\Mdeg(C) := \min \{ m\in \z \mid \mhF_m = \hC \}
$$
and
$$
\mdeg(C) := \min \{m \in \z \mid \mhF_m \neq 0 \}.
$$
(For the definition of $\{\mhF_m\}_{m \in \z}$, see Subsection~\ref{section tau}.)
Let $\{x_k\}_{1 \leq k \leq r}$ be a filtered basis for $C$.
The finiteness of the above values follows from the following lemma.
\begin{lem}
\label{deg basis}
The equalities
$$
\Mdeg(C)= \max_{1 \leq k \leq r}\{\Alex(x_k)-\Alg(x_k)\}
$$
and
$$
\mdeg(C)= \min_{1 \leq k \leq r}\{\Alex(x_k)-\Alg(x_k)\}
$$
hold.
\end{lem}
\begin{proof}
From the definition of $\{\mhF_m\}_{m \in \z}$, we can see that
$$
\mhF_m = \spanF\{U^{\Alg(x_k)}x_k \mid \Alex(x_k) -\Alg(x_k) \leq m \}.
$$
This completes the proof.
\end{proof}
\begin{cor}
\label{deg dual}
The equalities 
$$
\Mdeg(C^*)= - \mdeg(C) \text{ and } 
\mdeg(C^*)= - \Mdeg(C)
$$
hold.
\end{cor}
\begin{proof}
As shown in the proof of Proposition~\ref{dual}, 
we can take a filtered basis $\{x^*_k\}_{1 \leq k \leq r}$
such that 
$$
\Alex(x^*_k)= -\Alex(x_k) \text{ and } \Alg(x^*_k) = -\Alg(x_k).
$$
This completes the proof.
\end{proof}
Moreover, about the decomposition $\{C_{(i,j)}\}_{(i,j) \in \z^2}$ induced by 
a filtered basis $\{x_k\}$, 
we have the following lemma.
\begin{lem}
\label{supp}
The support 
$
\{(i,j) \mid C_{(i,j)} \neq 0  \}
$
is contained in the set
$$
\{\mdeg(C) \leq j-i \leq \Mdeg(C) \}.
$$
\end{lem}
\begin{proof}
If $U^l x_k$ is lying in $C_{(i,j)}$, then 
$$
\Alex(x_k) -\Alg(x_k)= \Alex(U^l x_k) - \Alg(U^l x_k) = j-i. 
$$
Therefore, by Lemma~\ref{deg basis}, we have $\mdeg(C) \leq j-i \leq \Mdeg(C)$.
\end{proof}
For a coordinate $(k,l) \in \z^2$, set
$$
R_{(k,l)}:=\{(i,j) \in \z^2 \mid i \leq k \text{ and } j \leq l \},
$$
and then $R_{(k,l)} \in \CR$. For any subset $S \subset \z^2$, 
define the {\it closure} of $S$ by
$$
cl(S) := \bigcup_{(i,j) \in S} R_{(i,j)}.
$$
Then we also have $cl(S) \in \CR(\z^2)$. 
In addition,  the equality 
$$
cl(S) = \bigcap_{R \in \CR(\z^2), S \subset R} R
$$ 
holds.
For any $R \in \CR(\z^2)$ and $m,M \in \z$ with $m \leq M$, define 
$$
S_{m,M}^R := \{(i,j)\in R \mid m \leq j-i \leq M  \}.
$$
Then, as a corollary of Lemma~\ref{supp}, we have the following.
\begin{cor}
\label{red}
For any formal knot complex $C$ and $R \in \CR(\z^2)$, the equality
$$
C_R = C_{cl(S_{\mdeg(C),\Mdeg(C)}^R)}
$$
holds.
\end{cor}
\begin{proof}
Since $R \supset cl(S_{\mdeg(C),\Mdeg(C)}^R)$, 
obviously we have $C_R \supset C_{cl(S_{\mdeg(C),\Mdeg(C)}^R)}$.
Next we prove the converse. Fix a filtered basis $\{x_k\}_{1 \leq k \leq r}$
and denote the induced decomposition by $\{C_{(i,j)}\}$.
By Lemma~\ref{subcpx decomp},
it suffices to show that for any $(i,j) \in R \setminus cl(S_{\mdeg(C),\Mdeg(C)}^R)$,
the equality $C_{(i,j)}=0$ holds.
Indeed, 
for any such coordinate $(i,j)$, at least one of the
inequalities
$$j-i < \mdeg(C) \text{ and } j-i > \Mdeg(C)$$ 
holds.
Therefore, it follows from Lemma~\ref{supp} that $C_{(i,j)}=0$.
\end{proof}
\subsubsection{Genus of a formal knot complex}
Now we define the {\it genus of a formal knot complex $C$} by
$$
g(C) := \max\{\Mdeg(C), -\mdeg(C) \}.
$$ 
Then it is obvious that $g(C) \geq 0$, and Corollary~\ref{deg dual} gives
$g(C^*)=g(C)$.
Moreover, for knot complexes, we have the following.
\begin{thm}[ \text{\cite{OS04genus}, \cite[Section~5]{OS04knot}}]
\label{knot genus}
For any knot $K$, the equality
$$
g(K)=\min \{g(C) \mid C \in [C^K]  \}
$$
holds.
\end{thm}
Moreover,
by definition, we have $-g(C) \leq \mdeg(C) \leq \Mdeg(C) \leq g(C)$.
Hence Corollary~\ref{red} gives the following.

\begin{cor}
\label{genus red}
For any formal knot complex $C$ and $R \in \CR(\z^2)$, the equality
$$
C_R = C_{cl(S^R_{-g(C), g(C)})}
$$
holds.
\end{cor}
The following lemma is useful for reducing $C_R$ 
in concrete situations.
\begin{lem}
\label{useful red}
The following assertions hold:
\begin{enumerate}
\item For any $k \in \z$, we have $C_{\{i \leq k\}} = C_{R_{(k,g(C)+k)}}$.
\item For any $l \in \z$, we have $C_{\{ j \leq l\}} = C_{R_{(g(C)+l,l)}}$.
\end{enumerate}
\end{lem}
\begin{proof}
Here we verify the assertion~(1). 
For any $k \in \z$,
we see
\begin{eqnarray*}
S^{\{ i \leq k\}}_{-g(C),g(C)}&=&\{i \leq k\} \cap \{-g(C) \leq j-i \leq g(C)\} \\
&\subset& \{i \leq k\} \cap \{j \leq  g(C)+i\}\\
&\subset& \{i \leq k\} \cap \{j \leq  g(C)+k\}= R_{(k,g(C)+k)}.
\end{eqnarray*}
Therefore, we have $\{i \leq k\} \supset (R_{(k,g(C)+k)}) \supset  cl(S^{\{i \leq k\}}_{-g(C),g(C)})$,
and hence Corollary~\ref{genus red} gives $C_{\{ i \leq k\}} = C_{R_{(k,g(C)+k)}}$.
Similarly, we can verify the assertion~(2).
\end{proof}
\subsection{Comparison with $[(T_{2,2g+1})^*]_{\np}$}
For $g \in \z_{\geq 0}$, let $T_{2,2g+1}$ be the $(2,2g+1)$-torus knot.
These knots are alternating knots such that $\sigma(T_{2,2g+1})=-2g$,
and hence it follows from Theorem~\ref{alt} that
$[T_{2,2g+1}]= g[T_{2,3}]$.
In this subsection, we consider comparing $\np$-classes with $[(T_{2g+1})^*]_{\np}$. 
First, we recall that the knot complex
$C^{(T_{2,2g+1})^*}$ has a filtered basis
$$
\{a_k, b_l \mid 0 \leq k \leq g, 0 \leq l \leq g-1\}
$$
satisfying:
\begin{itemize}
\item $\gr(a_k)=0$ and $\gr(b_l)= -1$.
\item $(\Alg(a_k), \Alex(a_k))=(-g+k, -k)$ and $(\Alg(b_l), \Alex(b_l))=(-g+l, -l-1)$.
\item $\partial a_k =b_{k-1} + b_{k}$ and $\partial b_l = 0$, where $b_{-1}=b_g=0$.\end{itemize}
Here we note that $a := a_0 + \cdots + a_g$ is a unique homological generator of 
$C^{(T_{2,2g+1})^*}$.
For any $g \in \z_{\geq 0}$, define
$$
R^g := \bigcup_{0 \leq n \leq g} R_{(-g+n, -n)}.
$$
Then we have the following sufficient condition for satisfying the inequality 
$[C]_{\np} \leq [(T_{2,2g+1})^*]_{\np}$.
\begin{prop}
\label{compare}
For any formal knot complex $C$, if $C_{R^{g}}$ contains a homological generator,
then the inequality 
$$
[C]_{\nu^+} \leq [(T_{2,2g+1})^*]_{\np}
$$
holds.
\end{prop}

\begin{proof}
Fix a filtered basis and denote the induced decomposition
by $\{C_{(i,j)}\}$.
Define the subsets $S_k \subset \z^2$ ($k=0,1, \ldots, g$) by 
$$
S_0 := R_{(-g,0)}
$$ 
and
$$
S_k := \{ i= -g+k, j \leq -k  \}
$$
for $1 \leq k \leq g$.
Then $R^{g}= \amalg_{0 \leq k \leq g} S_k$,
and hence we can uniquely decompose a homological generator $z \in C_{R^{g}}$
into a linear combination
$z = \sum_{k=0}^g z_k$, where $z_k \in \bigoplus_{(i,j)\in S_k}C_{(i,j)}$.
We denote $y_l := \partial (z_0 + \ldots + z_l)$ for any $0 \leq l \leq g-1$.
\begin{claim}
\label{y_l}
$y_l$ is lying in $C_{R_{(-g+l,-l-1)}}$.
\end{claim}
\begin{proof}
Since $z$ is a cycle, we see  
$y_l = \partial(z_0 + \cdots + z_l)=\partial(z_{l+1} + \cdots + z_{g})$.
Moreover, since the relations
$$\bigcup_{0\leq k \leq l} S_k = \bigcup_{0 \leq k \leq l}R_{(-g+k,-k)}$$
and
$$\bigcup_{l+1 \leq k \leq g} S_k \subset \bigcup_{l+1 \leq k \leq g}R_{(-g+k,-k)}$$
hold, we have 
$y_l \in C_{\bigcup_{0 \leq k \leq l}R_{(-g+k,-k)}} \cap 
C_{\bigcup_{l+1 \leq k \leq g}R_{(-g+k,-k)}}$. Here, Lemma~\ref{subcpx decomp}
gives
$$
\begin{array}{l}
\displaystyle
C_{\bigcup_{0 \leq k \leq l}R_{(-g+k,-k)}} \cap 
C_{\bigcup_{l+1 \leq k \leq g}R_{(-g+k,-k)}} \\
\ \\
\displaystyle
= 
(\bigoplus_{(i,j) \in \bigcup_{0 \leq k \leq l}R_{(-g+k,-k)}}C_{(i,j)})
\cap
(\bigoplus_{(i,j) \in \bigcup_{l+1 \leq k \leq g}R_{(-g+k,-k)}}C_{(i,j)})\\
\ \\
=C_{(\bigcup_{0 \leq k \leq l}R_{(-g+k,-k)}) \cap (\bigcup_{l+1 \leq k \leq g}R_{(-g+k,-k)})}
=C_{R_{(-g+l, -l-1)}}.
\end{array}
$$
\end{proof}
Now, we define a $\Lambda$-linear map $f \colon C^{(T_{2,2g+1})^*} \to C$
by
$$
f a_k = z_k \text{ and } f b_l = y_l.
$$
Then we can check that $f$ is a chain map over $\Lambda$.
(Notice that $\partial z_k = \partial (z_0 + \cdots + z_{k-1}) + \partial (z_0 + \cdots + z_{k})
= y_{k-1} + y_k$.)
Moreover, by Claim~\ref{y_l}, we have 
$$
(\Alg(fa_k), \Alex(fa_k)) \leq (-g+k,-k)
$$ 
and
$$
(\Alg(f b_l), \Alex(f b_l)) \leq (-g+l,-l-1).
$$
In addition, $f(a)=f(a_0+ \cdots + a_g) =z_0+\cdots + z_g =z$.
Now, Lemma~\ref{quasi-isom basis} proves that $f$ is a 
$\z^2$-filtered quasi-isomorphism.
\end{proof}

\subsection{An estimate of genus one complexes}
Here, we consider an estimate for genus one formal knot complexes.
\begin{thm}
\label{alg tau}
Let $C$ be a formal knot complex with $g(C)=1$.
\begin{enumerate}
\item
If $\tau(C)=1$, then $[C]_{\np} \geq [T_{2,3}]_{\np}$.
\item
If $\tau(C)=0$, then $[C]_{\np}= 0$.
\item
If $\tau(C)=-1$, then $[C]_{\np} \leq -[T_{2,3}]_{\np}$. 
\end{enumerate}
\end{thm}
\begin{proof}
By Lemma~\ref{tau homol gen}, we have a homological generator 
lying in $C_{\{i \leq -1\} \cup R_{(0,\tau(C))}}$.
Moreover, Lemma~\ref{add} and Lemma~\ref{useful red} imply that
$$
C_{\{i \leq -1\} \cup R_{(0,\tau(C))}}=
C_{\{i \leq -1\}} + C_{R_{(0,\tau(C))}}
=C_{R_{(-1,0)}} + C_{R_{(0,\tau(C))}}
=C_{R_{(-1,0)}\cup R_{(0,\tau(C))}}.
$$
As a result, 
we have a homological generator in $C_{R_{(-1,0)}\cup R_{(0,\tau(C))}}$.

First, suppose that $\tau(C)=0$. Then 
$C_{R_{(-1,0)}\cup R_{(0,\tau(C))}} = C_{R_{(0,0)}}$.
This proves $\np(C)=0$. Moreover, since $\tau(C^*)= - \tau(C)=0$ and 
$g(C^*)=g(C)=1$, we also have
$\np(C^*)=0$. Therefore, the assertion~(2) holds.

Next, suppose that $\tau(C)=-1$. Then  
$C_{R_{(-1,0)}\cup R_{(0,\tau(C))}} = C_{R_{(-1,0)}\cup R_{(0,-1)}}=C_{R^1}$.
Therefore, it follows from Proposition~\ref{compare}
that $[C]_{\np} \leq [(T_{2,3})^*]= -[T_{2,3}]$, and the assertion~(3) holds.

Finally, the assertion~(1) follows from the assertion~(3) and the fact
that $\tau(C^*) = -\tau(C)=-1$ and $[C^*]_{\np} =-[C]_{\np}$.
\end{proof}
Now we can prove the main theorem.
\setcounter{section}{1}
\setcounter{thm}{1}
\begin{thm}
For any knot $K$ with $g(K)=1$,
we have
$$
[K]_{\nu^+}
=
\begin{cases}
\left[ T_{2,3} \right]_{\nu^+} & \text{if } \tau(K)=1\\
\left[\text{{\rm unknot}}\right]_{\nu^+}=0 & \text{if } \tau(K)=0\\
\left[ (T_{2,3})^* \right]_{\nu^+}=-[T_{2,3}]_{\nu^+} & \text{if } \tau(K)=-1
\end{cases}.
$$
In other words, any genus one knot is $\nu^+$-equivalent
to one of the trefoil, its mirror and the unknot.
\end{thm}
\setcounter{section}{4}
\setcounter{thm}{9}

\begin{proof}
Let $K$ be a genus one knot.
Then, by Theorem~\ref{geom}, we have
$$
-[T_{2,3}]_{\np} \leq [K]_{\np} \leq [T_{2,3}]_{\np}.
$$
Moreover, by Theorem~\ref{knot genus}, we can take a knot complex $C^K$ 
with $g(C^K)=1$.
Hence, Theorem~\ref{alg tau} gives
$$
[K]_{\nu^+}
\begin{cases}
\geq \left[ T_{2,3} \right]_{\nu^+} & \text{if } \tau(K)=1\\
=0 & \text{if } \tau(K)=0\\
\leq -[T_{2,3}]_{\nu^+} & \text{if } \tau(K)=-1
\end{cases}.
$$
This completes the proof.
\end{proof}

\subsection{An estimate using $\Upsilon$}
Here we show an estimate which is obtained by using $\Upsilon$.

\begin{thm}
\label{alg Upsilon}
If $\Upsilon_C(1) = g(C)$, then
$[C]_{\np} \leq -g(C)[T_{2,3}]_{\np}$.
\end{thm}
\begin{proof}
By the definition of $\Upsilon$,
we have a homological generator lying in 
$C_{\{i+j \leq -g(C)\}}$.
Here, we note that 
$$
\{i+j \leq -g(C)\} \subset \{i \leq -g(C) \} \cup R^{g(C)} \cup \{j \leq -g(C)\}.
$$
Moreover,
Lemma~\ref{add} and Lemma~\ref{useful red} imply that
\begin{eqnarray*}
C_{\{i \leq -g(C) \} \cup R^{g(C)} \cup \{j \leq -g(C)\}}
&=&
C_{\{i \leq -g(C) \}}+C_{R^{g(C)}} + C_{\{j \leq -g(C)\}}\\
&=&
C_{R_{(-g(C),0)}}+C_{R^{g(C)}} + C_{R_{(0, -g(C))}} = C_{R^{g(C)}}.
\end{eqnarray*}
As a result, we have a homological generator in $C_{R^{g(C)}}$.
Therefore, Proposition~\ref{compare} proves that
$[C]_{\np} \leq [(T_{2,2g(C)+1})^*]_{\np} = -g(C) [T_{2,3}]_{\np}$.
\end{proof}

Now we can prove the following discriminant.

\setcounter{section}{1}
\setcounter{thm}{5}
\begin{thm}
The equality $[K]_{\np}= -g(K)[T_{2,3}]_{\nu^+}$ holds
if and only if $\Upsilon_{K}(1)= g(K)$.
\end{thm}
\setcounter{section}{4}
\setcounter{thm}{10}
\begin{proof}
The only-if-part obviously holds.
We prove the if-part.
For any knot $K$, by Theorem~\ref{geom}, we have
$$
[K]_{\np} \geq -g(K)[T_{2,3}]_{\np}.
$$
Moreover, by Theorem~\ref{knot genus}, we can take a knot complex $C^K$ 
with $g(C^K)=g(K)$.
Hence, if $\Upsilon_K(1)=g(K)$, then Theorem~\ref{alg Upsilon} gives
$$
[K]_{\nu^+} \leq -g(K)[T_{2,3}]_{\np}.
$$
This completes the proof.
\end{proof}

\section{New concordance invariants}
In this section, we discuss new invariants $\{\G_n\}$ of $\np$-classes whose values
are finite subsets of $\CR(\z^2)$.
\subsection{The invariants $\tG_0$ and $\G_0$}
As seen in Subsection~\ref{section invariants}, 
many invariants introduced in previous work  can be translated into the words
of closed regions containing a homological generator. From the view point,
it is natural to consider the universal set
$$
\tG_0(C) := \{R \in \CR(\z^2) \mid C_R \text{ contains a homological generator} \}.
$$
In fact, it behaves very naturally in terms of filtered quasi-isomorphism.
\begin{thm}
\label{tG_0 ineq}
If $[C]_{\np} \leq [C']_{\np}$, then $\tG_0(C) \supset \tG_0(C')$.
\end{thm}
\begin{proof}
By assumption, we have a $\z^2$-filtered quasi-isomorphism 
$f \colon C' \to C$. 
Therefore, for any element $R \in \tG_0(C')$ and a homological generator 
$x \in C'_R$, we see that 
$C_R$ also contains a homological generator $f(x)$, and hence 
$R \in \tG_0(C)$.
\end{proof}

As a corollary, we have the invariance of $\tG_0$.
Here $\mathcal{P}(\CR(\z^2))$ denotes the power set of $\CR(\z^2)$.
\begin{cor}
\label{tG_0 invariance}
$\tG_0(C)$ is invariant under $\np$-equivalence. In particular, 
$$\tG_0 \colon [C]_{\np} \mapsto \tG_0(C)$$ 
is a well-defined map
$
\mCf \to \mathcal{P}(\CR(\z^2)).
$
\end{cor}

By definition, $\tG_0(C)$ obviously has the following property.
\begin{prop}
\label{tG prop}
For any $R \in \tG_0(C)$ and $R' \in \CR(\z^2)$, if $R \subset R'$,
then $R' \in \tG_0(C)$.
\end{prop}
In particular, we see that $\tG_0(C)$ is an infinite set.
To extract an essential part of $\tG_0$,
we consider the minimalization of $\tG_0$.

For a subset $\mathcal{S} \subset \CR(\z^2)$,
an element $R \in \mathcal{S}$ is {\it minimal in $\mathcal{S}$}
if it satisfies
$$
\text{if } R' \in \mathcal{S} \text{ and } R' \subset R, \text{ then } R'=R.
$$
Define the map 
$$
\min \colon \mathcal{P}(\CR(\z^2)) \to \mathcal{P}(\CR(\z^2))
$$
by
$$
\mathcal{S} \mapsto \{ R \in \mathcal{S} \mid R \text{ is minimal in } \mathcal{S}\}.
$$ 
Now we define $\G_0(C)$ by 
$$
\G_0(C) := \min \tG_0(C).
$$
The invariance of $\G_0(C)$ under 
$\nuplus$ immediately follows from 
Corollary~\ref{tG_0 invariance}.

Here, for referring later, we prove the following lemma.
\begin{lem}
\label{min of finite}
Let $\mathcal{S} \subset \CR(\z^2)$ be a non-empty finite subset.
Then, for any $R \in \mathcal{S}$, there exists an element 
$R' \in \min \mathcal{S}$ with $R' \subset R$.
In particular, $\min \mathcal{S}$ is non-empty.
\end{lem}
\begin{proof}
We prove the lemma by the induction of the order of $S$.
If $|S|=1$, then $\min \mathcal{S} = \mathcal{S}$, and the assertion obviously holds.

Assume that for any subset of $\CR(\z^2)$ with order $n$, the assertion holds.
Let $\mathcal{S} \subset \CR(\z^2)$ be a subset
with order $n+1$. If any element of $\mathcal{S}$ is minimal in $\mathcal{S}$,
then the assertion holds for $\mathcal{S}$. 
Suppose that there exist elements 
$R,R' \subset \mathcal{S}$ such that $R' \subsetneq R$.
Then, since $\mathcal{S} \setminus \{ R\}$ has order $n$, the assertion holds for
$\mathcal{S} \setminus \{ R\}$. In particular, we have an element $R'' \in \min (\mathcal{S} \setminus \{R\})$
with $R'' \subset R'$. Here we note that  
$R \not\subset R''$, and hence $R'' \in \min \mathcal{S}$.
Moreover, we have $R'' \subset R' \subset R$.
This implies that the assertion holds for $\mathcal{S}$,
and completes the proof.
\end{proof}
\subsection{Finiteness of $\G_0$}
In this subsection, we show that $\G_0(C)$ is a finite set for any formal knot complex $C$.

\subsubsection{The region of a chain}
For a non-zero element $p=p(U) \in \Lambda$,
denote the lowest degree of $p$ by $l(p)$.
Let $C$ be a formal knot complex, and 
$\{x_k\}_{1 \leq k \leq r}$ a filtered basis for $C$.
For any non-zero chain $x= \sum_{1 \leq k \leq r} p_k(U) x_k$,
 we define \textit{the region of} $x$ as 
$$
R_x := cl\left\{(\Alg(U^{l(p_k)}x_k), \Alex(U^{l(p_k)}x_k)) \  \middle| \  
\begin{array}{ll}
1 \leq k \leq r\\
p_k(U) \neq 0
\end{array}
\right\}.
$$
Then we see that $R_x \in \CR(\z^2)$ and $x \in C_{R_x}$.
The following lemma implies that $R_x$ does not depend on the choice of $\{x_k\}$.
\begin{lem}
\label{chain's region cl}
The equality
$$R_x = \bigcap_{R \in \CR, x \in C_R} R
$$
holds.
In particular, $x \in C_R$ if and only if $R_x \subset R$.
\end{lem}
\begin{proof}
It is obvious that $R_x \supset \bigcap_{R \in \CR, x \in C_R} R.$
We prove the converse. Let $\{C_{(i,j)}\}$ be the decomposition of $C$
induced by $\{x_k\}$, and take $R \in \CR(\z^2)$ with $x \in C_R$.
Then, since $C_R = \bigoplus_{(i,j) \in R}C_{(i,j)}$
and 
$$C_{(i,j)}= \spanF\{U^l x_k \mid (\Alg(U^{l}x_k), \Alex(U^{l} x_k)) =(i,j)\},$$ 
we see that
$$
((\Alg(U^{l(p_k)}x_k), \Alex(U^{l(p_k)}x_k))) \in R
$$
for any $k \in \{1, \ldots, r\}$ with $p_k(U) \neq 0$.
This completes the proof.
\end{proof}
\begin{lem}
\label{chain's region map}
For any $\z^2$-filtered chain map $f \colon C \to C'$ and $x \in C$,
we have $R_{f(x)} \subset R_x$.
\end{lem}
\begin{proof}
Since $x \in C_{R_{x}}$, we see 
$$
f(x) \in  f(C_{R_x}) \subset C'_{R_x}.
$$
Hence, Lemma~\ref{chain's region cl} proves $R_{f(x)} \subset R_x$.
\end{proof}

\subsubsection{The regions of homological generators}

For a formal knot complex $C$, define
$$
\tgen_0(C) := \{\text{homological generator of } C\},
$$
$$
\tG'_0(C) := \{R_x \in \CR(\z^2) \mid x \in \tgen_0(C) \},
$$
and
$$
\G'_0(C) := \min \tG'_0(C).
$$
In addition, for $R \in \G_0(C)$, set
$$
\gen_0(C;R) := \{x \in \tgen_0(C) \mid R_x= R \},
$$
and call $x \in \gen_0(C;R)$ {\it a realizer of $R$}.
Notice that since $\dim_{\F}C_0 < \infty$,
$C$ has finitely many homological generators, and hence
both $\tG'_0(C)$ and $\G'_0(C)$ are finite and non-empty.
Therefore, the following theorem implies the finiteness 
and non-emptiness of $\G_0(C)$.
\begin{thm}
\label{G'_0 thm}
The equality $\G_0(C)=\G'_0(C)$ holds.
\end{thm}
\begin{proof}
We first prove $\G_0(C) \supset \G'_0(C)$.
Note that since $x \in C_{R_x}$ for any homological generator $x$,
we have $\tG_0(C) \supset \tG'_0(C)$.
Take $R_x \in \G'_0(C)$, and suppose that
 $R \in \tG_0(C)$ and $R \subset R_x$.
Then, there exists a homological generator $x'$ in $C_R$,
and hence Lemma~\ref{chain's region cl} implies 
$R_{x'} \subset R \subset R_{x}$.
Here, since $R_{x'} \in \tG'_0(C)$ and $R_{x}$ is minimal in $\tG'_0(C)$,
we have $R_{x'}=R=R_{x}$. This proves $R_x \in \G_0(C)$, and hence 
$\G_0(C) \supset \G'_0(C)$.

Next we prove $\G_0(C) \subset \G'_0(C)$.
For a given element $R \in \G_0(C)$,
we first need to prove that
$R \in \tG'_0(C)$. Here, in a similar way to the above arguments, we see that
there exists a homological generator $x$ such that $R_x \subset R$.
Moreover, since $R_x$ is also in $\tG_0(C)$ and $R$ is minimal in $\tG_0(C)$,
we have $R=R_x \in \tG'_0(C)$.
Now, the minimality of $R$ in $\tG'_0(C)$ immediately follows from
the minimality in $\tG_0(C)$. 
Therefore, we have $R \in \G'_0(C)$, and hence $\G_0(C) \subset \G'_0(C)$.
\end{proof}

As a corollary, we have the following useful property of $\G_0(C)$.
\begin{cor}
\label{minimalize}
For any formal knot complex $C$ and $R \in \CR(\z^2)$,the following holds:
$$
R \in \tG_0(C) \Leftrightarrow \exists R' \in \G_0(C), R' \subset R.
$$
\end{cor}

\begin{proof}
By the definition of $\tG_0(C)$, there exists a homological generator $x \in C$
with $R_x \subset R$. Moreover, since $R_x \in \tG'_0(C)$ and 
$\tG'_0(C)$ is a non-empty finite set, Lemma~\ref{min of finite} 
gives an element $R' \in \G'_0(C)=\G_0(C)$
with $R' \subset R_x \subset R$.
\end{proof}

Here we also mention the relationship of $\G_0(C)$ to the partial order on $\mCf$.

\begin{prop}
If $[C]_{\np} \leq [C']_{\np}$, then for any $R' \in \G_0(C')$, there
exists an element $R \in \G_0(C)$ with $R \subset R'$.
\end{prop}
\begin{proof}
For any $R' \in \G_0(C')$, Theorem~\ref{tG_0 ineq} shows $R' \in \tG_0(C)$. 
Now, by Corollary~\ref{minimalize}, 
we have an element $R \in \G_0(C)$ with $R \subset R'$.
\end{proof}
\subsection{Higher invariants $\G_n$}
Here, we discuss higher invariants.
\subsubsection{The secondary invariant $\G_1$}
For a formal knot complex $C$, suppose that
$\G_0(C)$ has distinct two elements $R_1$ and $R_2$.
Under the hypothesis, we define the secondary invariant 
$\G_1(C;R_1, R_2)$ as follows.
First, set
$$
\tgen_1(C;R_1, R_2) := \{x \in C_1 \mid \exists z_i \in \gen_0(C;R_i), \partial x = z_1 + z_2\},
$$
and
$$
\tG'_1(C;R_1, R_2) := \{R_x \mid x \in \tgen_1(C; R_1, R_2) \}.
$$
Then we define $\G_1(C; R_1, R_2)$ by
$$
\G_1(C; R_1, R_2) := \min \tG'_1(C;R_1, R_2).
$$
Here, for $R \in \G_1(C; R_1, R_2)$, we also define {\it the realizers of $R$} by
$$
\gen_1(C; R_1, R_2; R) := \{x \in \tgen_1(C; R_1, R_2) \mid R_x=R \}.
$$
Note that the above notions are independent of the order of $\{R_1, R_2\}$. 
\begin{lem}
\label{G_1 non-empty}
$\G_1(C; R_1, R_2)$ is a non-empty finite set.
\end{lem}
\begin{proof}
Take an arbitrary realizer $z_i \in \gen_0(C;R_i)$ for each $i =1,2$.
Then we see that $0 \neq [z_1]=[z_2] \in H_0(C) \cong \F$,
and hence there exists a 1-chain $x \in C_1$ such that $\partial x = z_1 + z_2$.
Moreover, $\dim_{\F}(C_1) < \infty$.
These facts shows that  $\tgen_1(C;R_1, R_2)$ is non-empty and finite.
Combining this fact with Lemma~\ref{min of finite}, 
we see that
$\G_1(C; R_1, R_2)$ is non-empty and finite.
\end{proof}

\begin{thm}
\label{G_1 ineq}
Suppose that $[C]_{\np} \leq [C']_{\np}$ and 
$\G_0(C) \cap \G_0(C')$ has distinct two elements $R_1$ and $R_2$.
Then, for any $R' \in \G_1(C'; R_1, R_2)$, there exists an element
$R \in \G_1(C; R_1, R_2)$ with $R \subset R'$.
\end{thm}
\begin{proof}
Take $z_i \in \gen_0(C';R_i)$ ($i=1,2$) and $x \in \gen_1(C';R_1, R_2; R')$ such that
$\partial x = z_1 + z_2$. Let $f \colon C' \to C$ be a $\z^2$-filtered quasi-isomorphism.
Then we see from the assumption and Lemma~\ref{chain's region map} that
$$
R_{f(z_i)}, R_i \in \tG'_0(C)
$$ 
and
$$
R_{f(z_i)} \subset R_{z_i} =R_i.
$$ 
Moreover, 
$R_i$ is minimal in $\tG'_0(C)$, and hence we have 
$R_{f(z_i)} = R_i$.
In particular, $f(z_i) \in \gen_0(C;R_i)$. 
Here, note that 
$$
\partial (f(x)) = f(\partial x) = f(z_1) + f(z_2),
$$
and hence $f(x) \in \tgen_1(C; R_1, R_2)$ and $R_{f(x)} \in \tG_1(C; R_1, R_2)$.
Now, Lemma~\ref{min of finite} and Lemma~\ref{chain's region map} give
an element $R \in \G_1(C; R_1, R_2)$ with
$$
R \subset R_{f(x)} \subset R_x=R'.
$$
\end{proof}
\begin{cor}
\label{G_1 invariance}
For any $[C]_{\np} \in \mCf$ and distinct two elements 
$R_1, R_2 \in \G_0(C)$, 
$\G_1(C;R_1,R_2) \in \mathcal{P}(\CR(\z^2))$ 
is an invariant of the $\np$-class $[C]_{\np}$. 
\end{cor}
\begin{proof}
Suppose that $[C]_{\np} = [C']_{\np}$.
Then, since $\G_0(C)=\G_0(C')$,  
we have $R_1, R_2 \in \G_0(C) \cap \G_0(C')$.
Let $R \in \G_1(C;R_1,R_2)$. Since $[C]_{\np} \geq [C']_{\np}$,
Theorem~\ref{G_1 ineq} gives 
$R' \in \G_1(C';R_1,R_2)$ with $R' \subset R$.
Moreover,  since $[C]_{\np} \leq [C']_{\np}$, 
we also have 
$R'' \in \G_1(C;R_1,R_2)$ with $R'' \subset R' \subset R$.
Here, since $R$ is minimal in $\tG'_1(C;R_1,R_2)$, we have
$$
R'' = R' = R,
$$
and hence $R=R' \in \G_1(C';R_1,R_2)$. This proves 
$\G_1(C;R_1,R_2) \subset \G_1(C';R_1,R_2)$. 

In the same way, we also have
$\G_1(C;R_1,R_2) \supset \G_1(C';R_1,R_2)$.
\end{proof}

\subsubsection{Higher invariants $\G_n$ with $n \geq 2$}
Now we construct more higher invariants $\G_n$ by induction.
Let $n$ be an integer with $n \geq 2$, and assume that 
$$
\begin{array}{lll}
\exists R^0_1, R^0_2 \in \G_0(C) &\text{ with }& R^0_1 \neq R^0_2,\\
\exists R^1_1, R^1_2 \in \G_1(C;\{R^0_1, R^0_2\}) &\text{ with }& R^1_1 \neq R^1_2,\\
\cdots &\ & \ \\
\exists R^{n-1}_1, R^{n-1}_2 \in 
\G_{n-1}(C; \{R^j_1, R^j_2\}_{j=0}^{n-2})
 &\text{ with }& R^{n-1}_1 \neq R^{n-1}_2.\\
\end{array}
$$
Then, we define 
\begin{eqnarray*}
\tgen_n(C;\{R^j_1, R^j_2\}_{i=0}^{n-1}) &:= &
\left\{x \in C_n \  \middle| \  
\begin{array}{r}
\exists z_i \in 
\gen_{n-1}(C;\{R^j_1,R^j_2\}_{j=0}^{n-2};R^{n-1}_i) \\
\text{ s.t.\ }
\left\{
\begin{array}{l}
\partial z_1 = \partial z_2 \\
\partial x = z_1 + z_2
\end{array}
\right.
\end{array}
\right\},\\
\tG'_n(C;\{R^j_1, R^j_2\}_{j=0}^{n-1}) &:= &\{R_x \mid x \in 
\tgen_n(C; \{R^j_1, R^j_2\}_{j=0}^{n-1}) \}, \text{ and }\\
\G_n(C; \{R^j_1, R^j_2\}_{j=0}^{n-1}) &:=& \min \tG'_n(C;\{R^j_1, R^j_2\}_{j=0}^{n-1}).\\
\end{eqnarray*}
In addition, for $R \in \G_n(C; \{R^j_1, R^j_2\}_{j=0}^{n-1})$, we define
$$
\gen_n(C; \{R^j_1, R^j_2\}_{j=0}^{n-1}; R) := 
\{x \in \tgen_n(C; \{R^j_1, R^j_2\}_{j=0}^{n-1}) \mid R_x=R \}.
$$
Unlike the cases of $\G_0$ and $\G_1$,
the author doesn't know whether  
$\G_n(C; \{R^j_1, R^j_2\}_{j=0}^{n-1})$ is empty or not,
while we see that it is finite.
(This is caused from the condition $\partial z_1 = \partial z_2$.)
However, if $\G_n(C; \{R^j_1, R^j_2\}_{j=0}^{n-1})$ is non-empty,
then we can show that it is invariant under $\np$-equivalence.
(As a consequence, the emptiness of $\G_n$ is also an invariant
of $\np$-classes.)

\begin{thm}
\label{G_n ineq}
Suppose that $[C]_{\np} \leq [C']_{\np}$ and the intersection
$$\G_{k}(C;\{R^j_1,R^j_2\}_{j=0}^{k-1}) \cap \G_{k}(C';\{R^j_1,R^j_2\}_{j=0}^{k-1})$$ 
has distinct two elements $R^k_1$ and $R^k_2$
(where $k=0,1, \ldots, n-1$, and $\{R^{j}_1, R^{j}_2\}_{j=0}^{-1}= \emptyset$).
Then, for any $R' \in \G_n(C'; \{R^j_1,R^j_2\}_{j=0}^{k-1})$, there exists an element
$R \in \G_n(C; \{R^j_1,R^j_2\}_{j=0}^{k-1})$ with $R \subset R'$.
In particular, the non-emptiness of 
$\G_n(C'; \{R^j_1,R^j_2\}_{j=0}^{k-1})$ implies the non-emptiness of
$\G_n(C; \{R^j_1,R^j_2\}_{j=0}^{k-1})$.
\end{thm}
\begin{proof}
The proof follows from arguments exactly the same as the proof of 
Theorem~\ref{G_1 ineq}. (We only need to care about the fact that
$R_{f(z_i)} \in \tG'_{n-1}(C; \{R^j_1, R^j_2\}_{j=0}^{n-2})$,
but this also can be proved by induction.)
\end{proof}
\begin{cor}
\label{G_n invariance}
For any $[C]_{\np} \in \mCf$ and 
sequence of distinct two elements 
$R^k_1, R^k_2 \in \G_k(C;\{R^j_1, R^j_2\}_{j=0}^{k-1})$ 
$(k=0,1, \ldots, n-1)$, 
$\G_n(C;\{R^j_1,R^j_2\}_{j=0}^{n-1}) \in \mathcal{P}(\CR(\z^2))$ 
is an invariant of the $\np$-class $[C]_{\np}$. 
\end{cor}
\begin{proof}
The proof follows from arguments exactly the same as the proof of 
Corollary~\ref{G_1 invariance}. 
(In fact, we only need to replace some symbols suitably.)
\end{proof}

\subsection{Relationship to other invariants}
In this subsection, we study the relationship of the new invariants 
$\G_0$ and $\G_1$ to the invariants reviewed in Subsection~\ref{section invariants}.
\subsubsection{Relationship of $\G_0$ to $\np$, $V_k$, $\tau$ and $\Upsilon$}
We first discuss the relationship of $\G_0$ to $\np$.
Here, recall $R_{(k,l)} := \{(i,j) \in \z^2 \mid (i,j) \leq (k,l) \}$.
\begin{prop}
\label{G_0 and np}
For any formal knot complex $C$, the values $\np(C)$ and $\np(C^*)$ are determined from
$\G_0(C)$ by the formulas
$$
\np(C)= \min\{m \in \z_{\geq 0} \mid \exists R \in \G_0(C), R \subset R_{(0,m)} \}
$$
and
$$
\np(C^*)= \min\{ m \in \z_{\geq 0} \mid \forall R \in \G_0(C), R \supset 
R_{(0,-m)} \}.
$$
\end{prop}
\begin{proof}
We can see that the equality
$$
\np(C)= \min \{m \in \z_{\geq 0} \mid R_{(0,m)} \in \tG_0(C) \}
$$
holds. Therefore, the first assertion immediately follows from Corollary~\ref{minimalize}.

Next, by Lemma~\ref{stb lem1}, the inequality
$\np(C^*) > m$ holds if and only if there is a homological generator $x \in C$ 
with $R_x \subset \{i \leq -1 \text{ or } j \leq -m-1 \}$.
Here, we note that 
$R_x \subset \{i \leq -1 \text{ or } j \leq -m-1 \}$ if and only if 
$R_x \not\supset  R_{(0,-m)}$.
Therefore, we have
$$
\np(C)= \min \{m \in \z_{\geq 0} \mid \forall R_x \in \tG'_0(C), R_x \supset R_{(0,-m)} \}.
$$
Moreover, Lemma~\ref{min of finite} implies
that any $R_x \in \tG'_0(C)$ includes $R_{(0,-m)}$
if and only if any $R_x \in \G_0(C)$ includes $R_{(0,-m)}$.
This completes the proof.
\end{proof}

From Proposition~\ref{G_0 and np}, we see that 
$\G_0$ detects the zero element as a $\np$-class. 
\begin{thm}
\label{detect zero}For any formal knot complex $C$,
the following holds:
$$
[C]_{\np} = 0 \Leftrightarrow \G_0(C) = \{ R_{(0,0)}\}
$$
\end{thm}
\begin{proof}
By the invariance of $\G_0$ under $\nuplus$ and easy computation 
$\G_0(\Lambda)=\{R_{(0,0)}\}$, the only-if-part obviously holds.
Moreover. the if-part immediately follows from Proposition~\ref{G_0 and np},
since the unique element $R :=R_{(0,0)} \in \G_0(C)$ satisfies 
$R \subset R_{(0,0)}$ and $R \supset R_{(0,0)}$.
\end{proof}
On the other hand, we will see in Subsection~\ref{no realizing}
that $\G_0$ is not a perfect invariant of $\np$-classes.

We can also translate the invariants $V_k$, $\tau$ and $\Upsilon$ as follows:
\begin{eqnarray*}
V_k(C)&=&\min\{m \in \z_{\geq 0} \mid  R_{(m,k+m)} \in \tG_0(C) \}\\
\tau(C)&=&\min\{m \in \z \mid (\{ i \leq -1\} \cup R_{(0,m)})  \in \tG_0(C) \}\\
\Upsilon_C(t)&=&-2\left(\min\{s \in \R \mid R^t(s) \in \tG_0(C) \}\right)
\end{eqnarray*}
(Here, recall $R^t(s) := \{(i,j) \in \z^2 \mid (1-t/2)i + (t/2)j \leq s \}$.)
Therefore, we have the following formulas.
\begin{prop}
\label{G_0 and others}
For any formal knot complex $C$, the invariants $V_k(C)$, $\tau(C)$ and 
$\Upsilon_C(t)$ are determined from
$\G_0(C)$ by the formulas:
\begin{eqnarray*}
V_k(C)&=&\min\{m \in \z_{\geq 0} \mid  
\exists R \in \G_0(C), R \subset R_{(m,k+m)} \}\\
\tau(C)&=&\min\{m \in \z \mid 
\exists R \in \G_0(C), R \subset
(\{ i \leq -1\} \cup R_{(0,m)})   \}\\
\Upsilon_C(t)&=&-2 \left( \min\{s \in \R \mid \exists R \in \G_0(C), R \subset
 R^t(s)\} \right)
\end{eqnarray*}
\end{prop}

\subsubsection{Relationship of $\G_1$ to $\Upsilon^2$}
Next, we discuss the relationship of $\G_1$ to $\Upsilon^2$.
(Precisely, we compare $\G_1$ with $\upsilon^2$ rather than $\Upsilon^2$.)
Let 
$$
\G^{t\pm}_0(C) := \{R \in \G_0(C) \mid R \subset \mF^{t \pm \delta}_{\upsilon_C(t \pm \delta)}  \},
$$
and then we see that the inequality
$$
\mathcal{Z}^{\pm}_t(C) \supset \bigcup_{R \in \G^{t\pm}_0(C)}
\gen_0(C;R)
$$
holds for each sign. (Remark that it does not become the equality in general, since 
we might have $x \in \tgen_0(C)$ such that 
$R \subsetneq R_x \subset \mF^{t \pm \delta}_{\upsilon_C(t \pm \delta)}$
for some $R \in  \G^{t\pm}_0(C)$. Such $x$ is lying in 
$\mathcal{Z}^{\pm}(C)$ but not in the right-hand side.)
In particular, $\mathcal{Z}^-_t(C) \cap \mathcal{Z}^+_t(C) = \emptyset$
 only if 
$\G^{t-}_0(C) \cap \G^{t+}_0(C) = \emptyset$.

For any $t \in (0,2)$, we set
$$
\G_1^t(C)
:= \bigcup_{R^{\pm}\in \G^{t\pm}_0(C), R^-\neq R^+} 
\G_1(C;R^-, R^+).
$$
Then, we have the following inequality.
(In light of the inequality, we can regard $\upsilon^2_{C,t}$ as a linear approximation of $\G^t_1(C)$.)
\begin{prop}
\label{G_1 and upsilon^2}
For any formal knot complex $C$, $t \in (0,2)$ and $s \in [0,2]$, 
the inequality
$$
\upsilon^2_{C,t}(s) \leq \min \{r \in \R \mid \exists R \in \G^t_1(C), R \subset
(R^t(\upsilon_C(t)) \cup R^s(r))\}.
$$
holds.
\end{prop}
\begin{proof}
Denote the right-hand side of the inequality in Proposition~\ref{G_1 and upsilon^2}
by $\upsilon^2_{\G^t_1(C)}(s)$.
Then, we can take $R \in \G^t_1(C)$ with 
$R \subset (R^t(\upsilon_C(t)) \cup R^s(\upsilon^2_{\G_1^t(C)}(s)))$.
Moreover, by the definition of $\G_1^t(C)$, there exists elements
$R^{\pm} \in \G^{t\pm}_0(C)$ such that $R^- \neq R^+$
and $R \in \G_1(C; R^-, R^+)$.
This implies that we have a homological generator
$$
z^{\pm} \in \gen_0(C;R^{\pm}) \subset \mathcal{Z}^{\pm}_t(C)
$$
for each sign
and 1-chain $x \in C_{R_x}\subset 
C_{R^t(\upsilon_C(t)) \cup R^s(\upsilon^2_{\G_1^t(C)}(s))}$
such that $\partial x = z^- - z^+$.
Here, by Lemma~\ref{add}, we see
$$
C_{R^t(\upsilon_C(t)) \cup R^s(\upsilon^2_{\G_1^t(C)}(s))}
=C_{R^t(\upsilon_C(t))} + C_{R^s(\upsilon^2_{\G_1^t(C)}(s))}
=\mF^t_{\upsilon_C(t)} + \mF^s_{\upsilon^2_{\G_1^t(C)}(s)},
$$
and hence $[z^-]-[z^+]= [\partial x]= 0 \in 
H_0(\mF^t_{\upsilon_C(t)} + \mF^s_{\upsilon^2_{\G_1^t(C)}(s)})$.
This shows the desired inequality
$\upsilon^2_{C,t}(s) \leq \upsilon^2_{\G_1^t(C)}(s)$.
\end{proof}

\subsection{Genus one complexes with no realizing knot}
\label{no realizing}
In this subsection, we define the complexes $C^n$ in Section~1 precisely, 
and prove Theorem~\ref{formal genus1} and Corollary~\ref{formal genus1 cor}.

For any $n \in \z_{>0}$, we define a $\F$-vector space $\bC^n$
with a basis $\{x_k, x'_k , y\}_{k=0}^{n-1}$
and $\F$-linear map $\bpartial \colon \bC^{n} \to \bC^{n}$ as follows:
\begin{itemize}
\item $\gr(x_k)=\gr(x'_k)=k$ and $\gr(y)=n$ ($0 \leq k \leq n-1$).
\item $\bpartial x_k = \bpartial x'_k = x_{k-1} + x'_{k-1}$ and 
$\bpartial y=x_{n-1}+x'_{n-1}$ ($0 \leq k \leq n-1$), 
where we define $x_{-1}+x'_{-1} :=0$.
\item $(\Alg(x_k), \Alex(x_k))=(k,k+1)$, $(\Alg(x'_k), \Alex(x'_k))=(k+1,k)$ and 
$(\Alg(y), \Alex(y))=(n,n)$ ($0 \leq k \leq n-1$).
\end{itemize}
Then we can check that $(\bC, \bpartial)$ satisfies all conditions of 
Lemma~\ref{construction}. 
(Figure~\ref{C^n} in Section~1 depicts the complex $(\bC^n, \bpartial)$.)
Therefore, we have a formal knot complex
$(C, \partial)$ which is related to $(\bC, \bpartial)$
as described in Lemma~\ref{construction}.
Note that $C^1$ coincides with the knot complex for the right-hand trefoil $T_{2,3}$.
Moreover, $g(C^n)=1$ for any $n$. 
\begin{prop}
\label{regions of C^n}
For any $n \in \z_{>0}$, the formal knot complex $C^n$ satisfies the following:
\begin{enumerate}
\item $\G_0(C^n)=\{R_{(0,1)}, R_{(1,0)}\}$. 
\item $\G_k(C^n;\{R_{(j,j+1)}, R_{(j+1,j)}\}_{j=0}^{k-1})=\{R_{(k,k+1)}, R_{(k+1,k)}\}$ 
$(1 \leq k \leq n-1)$. 
\item
$\G_n(C^n;\{R_{(j,j+1)}, R_{(j+1,j)}\}_{j=0}^{n-1})=\{R_{(n,n)}\}$. 
\end{enumerate}
\end{prop}
\begin{proof}
Obviously, we see that
$$\tgen_0(C^n) = \{x_0, x'_0\}$$ 
and 
$$\tG'_0(C^n)= \{R_{x_0}, R_{x'_0} \}=\{R_{(0,1)}, R_{(1,0)}\}.$$ 
Moreover,
both $R_{(0,1)}$ and $R_{(1,0)}$ are minimal in $\{R_{(0,1)}, R_{(1,0)}\}$,
and hence we have
$\G_0(C^n)=\{R_{(0,1)}, R_{(1,0)}\}$.

Next, fix $m \in \{0, 1, \ldots, n-2\}$, 
and assume that the assertion~(2)
holds for any $1 \leq k \leq m$. Then the equalities 
$$
\gen_m(C^n;\{R_{(j,j+1)}, R_{(j+1,j)}\}_{j=0}^{m-1};R_{(m,m+1)})=\{ x_{m}\}
$$
and
$$
\gen_m(C^n;\{R_{(j,j+1)}, R_{(j+1,j)}\}_{j=0}^{m-1};R_{(m+1,m)})=\{ x'_{m}\}
$$
must hold. Now we see 
$$
\tgen_{m+1}(C^n;\{R_{(j,j+1)}, R_{(j+1,j)}\}_{j=0}^{m})= \{x_{m+1}, x'_{m+1} \},
$$
and hence we can conclude
$$
\G_{m+1}(C^n;\{R_{(j,j+1)}, R_{(j+1,j)}\}_{j=0}^{m})=\{ R_{(m+1,m+2)}, R_{(m+2,m+1)}\}.
$$
This proves the assertion~(2).
Similarly, we can prove the assertion~(3).
\end{proof}

Now, we can easily prove the following theorems from the above computation.

\setcounter{section}{1}
\setcounter{thm}{7}
\begin{thm}
The $\np$-classes $\{[C^{n}]_{\np}\}_{n=1}^{\infty}$ are mutually distinct in $\mCf$,
while $\tau(C^n)=1$ for any $n$. In particular, the complement $\mF^f_1 \setminus \pi_{\np}(\mF_1)$ is infinite.
\end{thm}
\setcounter{section}{5}
\setcounter{thm}{18}
\begin{proof}
The first half assertion directly follows from Proposition~\ref{regions of C^n}.
Moreover, since the relations $\tau(k[T_{2,3}]_{\np})=k$, $\tau(C^n)=1$  and 
$[C^n]_{\np} \neq [C^1]_{\np} = [T_{2,3}]_{\np}$ hold
for any $k\in \z$ and $n \geq 2$,
we have $[C^n]_{\np} \neq k [T_{2,3}]_{\np}$.
This proves the second half assertion.
\end{proof}

\setcounter{section}{1}
\setcounter{thm}{8}
\begin{cor}
The formal knot complexes $\{C^{n}\}_{n = 2}^{\infty}$ 
cannot be realized by any knot in $S^3$.
\end{cor}
\setcounter{section}{5}
\setcounter{thm}{18}

\begin{proof}
If there exists a knot $K$ with $[C^K]=[C^n]$ for some $n \geq 2$, 
then it follows from Proposition~\ref{regions of C^n} and Theorem~\ref{knot genus}
that $\tau(K)=\tau(C^n)=1$ and $g(K)=1$. 
(Note that $1 = g(C^n) \geq \min \{g(C) \mid C \in  [C^n]=[C^K] \} =g(K) \geq \tau(K) =1$.)
Therefore, by Theorem~\ref{main thm}, we have 
$$
[C^n]_{\np} = [K]_{\np} = [T_{2,3}]_{\np}= [C^1]_{\np},
$$
which contradicts to Theorem~\ref{formal genus1}.
\end{proof}

\bibliographystyle{hplain}
\bibliography{tex_genus1}

\end{document}